\documentclass[11pt]{amsart}
\usepackage{amscd,amssymb}
\usepackage[arrow,matrix]{xy}
\usepackage[colorlinks,plainpages,backref,urlcolor=blue]{hyperref}

\newtheorem{theorem}[subsection]{Theorem}
\newtheorem{lemma}[subsection]{Lemma}
\newtheorem{prop}[subsection]{Proposition}
\newtheorem{corollary}[subsection]{Corollary}

\newtheorem{thm}{Theorem}

\theoremstyle{definition}
\newtheorem{remark}[subsection]{Remark}
\newtheorem{definition}[subsection]{Definition}
\newtheorem{example}[subsection]{Example}
\newtheorem*{ack}{Acknowledgments}

\numberwithin{equation}{section}
\newcommand{\cP}{{\mathcal P}}

\newcommand{\A}{{\mathcal A}}
\newcommand{\OO}{{\mathcal O}}

\newcommand{\CC}{{\mathcal C}}

\newcommand{\V}{{\mathcal V}}
\newcommand{\R}{{\mathcal R}}
\newcommand{\TT}{{\mathcal T}}
\newcommand{\QQ}{{\mathcal Q}}

\newcommand{\h}{\mathfrak {H}}
\newcommand{\m}{\mathfrak {m}}
\newcommand{\wh}{\widehat{\mathfrak {H}}}  
\newcommand{\bb}{\mathfrak {b}}
\newcommand{\gl}{\mathfrak {gl}}
\renewcommand{\sl}{\mathfrak {sl}}

\newcommand{\oM}{\overline{M}}
\newcommand{\oC}{\overline{C}}

\newcommand{\sV}{{\sf V}}
\newcommand{\sE}{{\sf E}}
\newcommand{\sW}{{\sf W}}
\newcommand{\kk}{\kappa}

\newcommand{\BB}{\mathbb{B}}
\newcommand{\T}{\mathbb{T}}
\newcommand{\Z}{\mathbb{Z}}
\newcommand{\N}{\mathbb{N}}
\newcommand{\Q}{\mathbb{Q}}
\newcommand{\RR}{\mathbb{R}}
\newcommand{\C}{\mathbb{C}}
\newcommand{\K}{\Bbbk}
\newcommand{\PP}{\mathbb{P}}
\newcommand{\FF}{\mathbb{F}}
\newcommand{\bF}{\mathbb{F}}
\newcommand{\bL}{\mathbb{L}}
\newcommand{\eexp}{\mathsf{exp}}

\DeclareMathOperator{\Hom}{Hom}
\DeclareMathOperator{\rank}{rank}
\DeclareMathOperator{\im}{im}

\DeclareMathOperator{\coker}{coker}
\DeclareMathOperator{\spn}{span}
\DeclareMathOperator{\id}{id}

\DeclareMathOperator{\tors}{Tors}
\DeclareMathOperator{\codim}{codim}

\DeclareMathOperator{\gr}{gr}
\DeclareMathOperator{\GL}{GL}
\DeclareMathOperator{\SL}{SL}
\DeclareMathOperator{\ab}{ab}
\DeclareMathOperator{\abf}{abf}
\DeclareMathOperator{\ad}{ad}
\DeclareMathOperator{\odd}{odd}
\DeclareMathOperator{\Zero}{Zero}

\DeclareMathOperator{\Rep}{Rep}
\DeclareMathOperator{\Sing}{Sing}
\DeclareMathOperator{\loc}{loc}
\DeclareMathOperator{\Lie}{Lie}
\DeclareMathOperator{\groups}{groups}

\newcommand{\surj}{\twoheadrightarrow}
\newcommand{\isom}{\xrightarrow{\,\simeq\,}}
\newcommand{\eqv}{{\Longleftrightarrow}}
\newcommand{\llangle}{{\langle\!\langle}}
\newcommand{\rrangle}{{\rangle\!\rangle}}
\newcommand{\abs}[1]{\left| #1 \right|}

\newenvironment{romenum}
{

\begin{enumerate}}{\end{enumerate}}

\begin{document}

\title[Topology and geometry of cohomology jump loci]%
{Topology and geometry of cohomology jump loci}

\author[A.~Dimca]{Alexandru Dimca$^{1,4}$}
\address{  Laboratoire J.A.~Dieudonn\'{e}, UMR du CNRS 6621,
               Universit\'{e} de Nice-Sophia-Antipolis,
               Parc Valrose,
               06108 Nice Cedex 02,
               France}
\email{dimca@math.unice.fr}

\author[\c{S}.~Papadima ]{\c{S}tefan Papadima$^{1,2}$}
\address{Institute of Mathematics ``Simion Stoilow",
P.O. Box 1-764,
RO-014700 Bucharest, Romania}
\email{Stefan.Papadima@imar.ro}

\author[A.~Suciu]{Alexander~I.~Suciu$^{3}$}
\address{Department of Mathematics,
Northeastern University,
Boston, MA 02115, USA}
\email{a.suciu@neu.edu}
\urladdr{http://www.math.neu.edu/\~{}suciu}

\thanks{$^1$Partially supported by the French-Romanian
Programme LEA Math-Mode}
\thanks{$^2$Partially supported by a grant of the
Romanian Ministry of Education and Research}
\thanks{$^3$Partially supported by NSF grant DMS-0311142
and NSA grant H98230-09-1-0012}
\thanks{$^4$Partially supported by grant ANR-08-BLAN-0317-02}



\begin{abstract}
We elucidate the key role played by formality in the 
theory of characteristic and resonance varieties. 
We define relative characteristic and resonance varieties,
$V_k$ and $R_k$, related to twisted group cohomology with 
coefficients of arbitrary rank. We show that the germs at 
the origin of $V_k$ and $R_k$ are analytically isomorphic,
if the group is $1$-formal; in particular, the tangent cone 
to $V_k$ at $1$ equals $R_k$. These new obstructions 
to $1$-formality lead to a striking rationality property
of the usual resonance varieties.  
A detailed analysis of the irreducible components of
the tangent cone at $1$ to the first characteristic 
variety yields powerful obstructions to realizing a 
finitely presented group as the fundamental group 
of a smooth, complex quasi-projective algebraic variety.  
This sheds new light on a classical problem of J.-P.~Serre. 
Applications to arrangements, configuration spaces, 
coproducts of groups, and Artin groups are given.
\end{abstract}

\maketitle

\tableofcontents

\section{Introduction} 
\label{sec=intro}

A recurring theme in algebraic topology and  
geometry is the extent to which the cohomology of a 
space reflects the underlying topological or geometric 
properties of that space.  In this paper, we focus on the 
degree-one cohomology jumping loci of the fundamental 
group $G$ of a connected CW-complex $M$: 
the characteristic varieties $\V_k(G)$,  the resonance 
varieties $\R_k(G)$, as well as relative versions of both.  
Our goal is to relate these two 
sets of varieties, and to better understand their structural 
properties, under certain naturally defined conditions 
on $M$.  In turn, this analysis yields new and powerful 
obstructions for a finitely generated group $G$ to be 
$1$-formal, or realizable as the fundamental 
group of a smooth, complex quasi-projective variety.

\subsection{Characteristic varieties}
\label{subsec:charvars}
Let $M$ be a connected CW-complex with finite $1$-skeleton, and
fundamental group $G=\pi_1(M)$. Let $\rho\colon \BB \to \GL (V)$
be a rational representation of the linear algebraic group $\BB$.
The {\em relative characteristic varieties}\/ associated to these 
data are the jump loci for twisted cohomology of $M$ with 
coefficients in $V$,
\begin{equation} 
\label{eq:charvarx}
\V^i_k(M, \rho)=\{\rho' \in \Hom_{\rm groups}(G, \BB) \mid 
\dim_{\C} H^i(M, {}_{\rho\rho'}V)\ge k\},
\end{equation}
where ${}_{\rho\rho'}V$ is the $G$-module defined by the 
representation $\rho\circ \rho'\colon G\to \GL(V)$.  As long as 
the $r$-skeleton of $M$ is finite, the sets $\V^i_k(M,\rho)$ 
are Zariski closed subsets of the representation variety 
$\Hom(G, \BB)$, for all $i\le r$ and $k\ge 1$. 

The usual characteristic varieties, denoted by $\V^i_k(M)$, 
correspond to the case when $\BB=\GL_1(\C)=\C^*$ and 
$\rho=\id_{\BB}$. These varieties 
emerged from the work of S.~Novikov \cite{N}
on Morse theory for closed $1$-forms on manifolds.  
Foundational results on the structure of the cohomology 
support loci for local systems on quasi-projective algebraic 
varieties were obtained by  Beauville \cite{Beau1, Beau2}, 
Green and Lazarsfeld \cite{GL}, Simpson \cite{Si92, Sim}, 
Arapura \cite{A}, and Campana \cite{Cam}.

Representation varieties have been intensively studied; 
see for instance \cite{GM, KM} and the references therein.
It seems natural to begin a systematic investigation 
of their filtration by relative characteristic varieties.
We consider here only the jump loci in degree $i=1$.%
\footnote{The higher degree cohomology jumping loci 
will be treated in a forthcoming paper.}
These loci depend solely on $G=\pi_1(M)$, 
so we write them as $\V_k(G, \rho)=\V^1_k(M, \rho)$. 
We work with arbitrary representations $\rho$, since 
this general approach provides a clearer 
conceptual framework for our results.  Moreover, as 
we note in Example \ref{ex:heisenberg group}, the 
relative characteristic varieties may well contain 
more information than the usual ones. 
When it comes to applications, though, we concentrate 
exclusively on the varieties $\V_k(G)=\V_k(G, \id_{\C^*})$, 
sitting inside the character variety $\T_G= \Hom (G, \C^*)$.  

\subsection{Resonance varieties}
\label{subsec:resvars}
Keeping the notation from above, let  $\bb$ the Lie 
algebra of $\BB$, and let $\theta =d_1(\rho)\colon \bb \to \gl (V)$ 
be the induced representation on tangent spaces. Denote 
by $V\rtimes_{\theta} \bb$ the corresponding semidirect 
product, with Lie algebra structure defined by 
$[(u,g),(v,h)]=(\theta(g)(v)-\theta(h)(u),[g,h])$.  
Let $H^{\bullet} M=H^{\bullet}(M, \C)$ be the cohomology 
ring of $M$, and consider the graded Lie algebra 
\begin{equation}
\label{eq:grlie}
H^{\bullet} M \otimes (V\rtimes_{\theta} \bb),
\end{equation}
with Lie bracket given by $[a \otimes y,b\otimes z]=ab \otimes [y,z]$.   
Note that the Lie identities are satisfied for \eqref{eq:grlie} only 
in the graded sense, according to the Koszul sign conventions; 
in particular, the Lie bracket  is symmetric on degree one elements.  

Now let $x$ be an element in $H^1 M\otimes \bb$ with the 
property that $[x,x]=0\in H^{2} M \otimes (V\rtimes_{\theta} \bb)$. 
As we show in \S\ref{ss34}, the restriction of the 
adjoint operator $\ad_x$ to the graded Lie ideal 
$H^{\bullet} M \otimes V$ yields a cochain complex, 
called the {\em relative Aomoto complex} of $H^{\bullet} M$ 
with respect to $x$,
\begin{equation} 
\label{eq=12relaom}
\xymatrix{(H^{\bullet} M\otimes V, \ad_x) : 
\  \cdots \ar[r] & H^i M \otimes V 
\ar^(.45){\ad_x}[r] & H^{i+1} M \otimes V \ar[r]& \cdots
}.
\end{equation}
The {\em relative resonance varieties} of $M$, associated to the 
representation $\theta \colon \bb \to \gl (V)$, are the jumping loci 
for the homology of this cochain complex:
\begin{equation} 
\label{eq:resvarx}
\R^i_k(M, \theta)= \{x \in H^1 M \otimes \bb 
\mid [x,x]=0 \ \text{and} \ 
\dim_{\C} H^i(H^{\bullet} M \otimes V,\ad_x) \ge  k\}.
\end{equation}
The loci $\R^1_k(M, \theta)$ depend only on $G=\pi_1(M)$, so 
we write them as $\R_k(G, \theta)$.  It is readily seen that 
$\R_k(G, \theta)$ is a homogeneous algebraic subvariety 
of the affine space $H^1(G,\C) \otimes \bb$, for each $k\ge 1$.

As mentioned before, we are mainly interested here 
in the case when $\bb=\C=\gl_1(\C)$ and $\theta=\id_{\bb}$. 
Then, as we show in Lemma \ref{lem=abaom}, every 
element $z\in H^1 M \cong H^1 M \otimes \bb$ satisfies 
$[z,z]=0$, and \eqref{eq=12relaom} becomes the usual 
Aomoto complex, $(H^{\bullet} M, \mu_z)$, where $\mu_z$ 
denotes left-multiplication by $z$.  Consequently, 
the relative resonance varieties $\R^i_k(M,\id_{\C})$ 
coincide with the usual resonance varieties $\R^i_k(M)$,  
first defined by Falk \cite{F} in the context of complex hyperplane 
arrangements, and further analyzed in \cite{CS2, MS1, LY, FY}.  
In the applications, we will focus on the varieties 
$\R_k(G):= \R_k(G, \id_{\C})$, sitting inside the 
affine space $H^1(G,\C)$.

\subsection{The tangent cone theorem}
\label{subsec=tangent}

The key topological property that allows us to relate the 
characteristic and resonance varieties of a space $M$ 
is {\em formality}, in the sense of D.~Sullivan \cite{S}.  
Since we deal solely with the fundamental group 
$G=\pi_1(M)$, we only need $G$ to be {\em $1$-formal}. 
This property requires that $E_G$, the Malcev 
Lie algebra of $G$ (over $\C$), be isomorphic, 
as a filtered Lie algebra, to the holonomy Lie algebra 
$\h(G)=\bL/\langle \im \partial_G\rangle$, completed 
with respect to bracket length filtration.  
Here $\bL$ denotes the free Lie algebra on $H_1(G,\C)$, 
while $\partial_G$ denotes the comultiplication map 
$H _2 (G, \C)\to  \bigwedge ^2H_1(G,\C)$. 

A group $G$ is $1$-formal if and only if $E_G$ can  
be presented with quadratic relations only; see 
Section \ref{sec=holo} for details. Many interesting 
groups fall in this class: knot groups \cite{S}, 
certain pure braid groups of Riemann surfaces 
\cite{B, Hai}, pure welded braid groups \cite{BP}, 
fundamental groups of compact K\"{a}hler 
manifolds \cite{DGMS}, fundamental groups 
of complements of hypersurfaces in $\C\PP^n$ 
\cite{K}, and finitely generated Artin groups \cite{KM} 
are all $1$-formal.

Our starting point is a result of Kapovich--Millson \cite{KM}, 
establishing an analytic isomorphism of germs of 
representation varieties,
\begin{equation}
\label{eq=repsp}
(\Hom_{\Lie} (\h(G), \bb), 0)\stackrel{\simeq}{\longrightarrow}
(\Hom_{\groups} (G, \BB), 1)\, ,
\end {equation} 
valid for all finitely generated, $1$-formal groups $G$.  
As we note in Lemma \ref{lem=repqc}, the representation 
space $\Hom_{\Lie} (\h(G), \bb)$ is naturally identified with
the quadratic cone $\{ x\in H^1G \otimes \bb \mid [x,x]=0 \}$. 
Our first main result reads as follows. 

\begin{thm} 
\label{thm=tcfintro} 
Let $G$ be a finitely generated, $1$-formal group, and let 
$\rho\colon \BB \to \GL(V)$ be a rational representation of 
a linear algebraic group, with differential 
$\theta \colon \bb \to \gl (V)$.  Then, for each $k\ge 1$, 
\begin{enumerate}

\item \label{aa1}
The isomorphism \eqref{eq=repsp} induces an  
analytic isomorphism of germs,
\[
(\R_k(G, \theta),0)  \xrightarrow{\,\simeq\,} (\V_k(G, \rho),1).
\]

\item \label{aa2}
This in turn induces an isomorphism between the tangent cone 
variety at the origin, $TC_1(\V_k(G, \rho))$, and the resonance 
variety $\R_k(G, \theta)$.

\item \label{aa3}
When  $\BB=\C^*$ and $\rho= \id_{\BB}$, the isomorphism 
\eqref{eq=repsp} is induced by the usual exponential 
map, $\exp\colon \Hom(G,\C) \to \Hom(G,\C^*)$. 
\end{enumerate}
\end{thm}

Theorem \ref{thm=tcfintro} sharpens and extends several  
results from \cite{ESV, STV, CS2}, which only 
apply to the case when $G$ is the fundamental group 
of the complement of a complex hyperplane arrangement.  
(Further information in the case of hypersurface arrangements 
can be found in \cite{Li01, DM}.)  The point is that only a 
{\em topological}\/ property---namely, $1$-formality---is needed 
for the conclusion of Theorem \ref{thm=tcfintro} to hold.  
A similar approach (in terms of the relative Malcev completion,
in the sense of Hain \cite{H93}) 
was used by Narkawicz \cite{Na} in the case of hyperplane 
arrangements.

For an arbitrary finitely presented group $G$, the tangent 
cone to $\V_k(G)$ at the origin is contained in $\R_k(G)$, 
see Libgober \cite{Li}.  Yet the inclusion may well be strict.  
In fact, as noted in Example  \ref{ex:notinj} and 
Remark \ref{rem:lib}, the tangent cone formula 
from Theorem \ref{thm=tcfintro} fails even for 
fundamental groups of smooth, quasi-projective varieties.  

Theorem \ref{thm=tcfintro} provides a new, and quite 
powerful obstruction to $1$-formality of groups---and thus, 
to formality of spaces.  As illustrated in Example \ref{ex=nonformal}, 
this obstruction complements, and in some cases strengthens, 
classical obstructions to ($1$-) formality, due to Sullivan, 
such as the existence of an isomorphism 
$\gr(G) \otimes \C \cong \h(G)$.   

\subsection{Further obstructions to $1$-formality}
\label{ss=more1f}

As a first application of Theorem \ref{thm=tcfintro}, we obtain 
the following result.

\begin{thm}
\label{thm=bnew}
Let $G$ be a finitely generated, $1$-formal group. Then:
\begin{enumerate}

\item \label{bn1}
All irreducible components of $\R_k(G)$ are linear subspaces of 
$H^1(G, \C)$, defined over $\Q$.

\item \label{bn2}
All irreducible components of $\V_k(G)$ containing $1$ 
are connected subtori of the character variety $\T_G=\Hom(G,\C^*)$.

\item \label{bn3}
Let $\{ \V_k^{\alpha} \}_{\alpha}$ be the irreducible components of 
$\V_k(G)$ containing $1$. Then their tangent spaces at the identity, 
$\{ T_1(\V_k^{\alpha}) \}_{\alpha}$, are the irreducible components 
of $\R_k(G)$.
\end{enumerate}
\end{thm}

The subtle linearity and rationality properties from \eqref{bn1} 
above reveal a striking phenomenon in non-simply connected rational 
homotopy theory: the existence of finite-dimensional, 
graded-commu\-tative $\Q$-algebras, which are {\em not}\/ 
realizable as cohomology rings of finite, formal CW-complexes; 
see Example \ref{ex=eucl}. (By \cite{S}, this cannot 
happen in the $1$-connected case.)

The {\em Alexander invariant}, $B_G= G'/G''$, viewed as a module 
over the group ring $\Z{G}_{\ab}$, is a classical object, 
extensively studied. In Theorem \ref{thm=complalex}, we give 
an explicit description of the $I$-adic completion of 
$B_G\otimes \Q$, involving only the rational holonomy 
Lie algebra of $G$. This  description is valid for an 
arbitrary finitely generated, $1$-formal group $G$.
 
\subsection{Serre's question}
\label{subsec=serre}

As is well-known, every finitely presented group $G$ is the 
fundamental group of a smooth, compact, connected 
$4$-dimensional manifold $M$.  Requiring a complex structure 
on $M$ is no more restrictive, as long as one is willing to go up 
in dimension; see Taubes \cite{Ta}. Requiring that $M$ be a 
compact K\"{a}hler manifold, though, puts extremely strong 
restrictions on what $G=\pi_1(M)$ can be.  We refer to \cite{ABC} 
for a comprehensive survey of K\"{a}hler groups.   

J.-P. Serre asked the following question: 
which finitely presented groups can be realized as 
fundamental groups of smooth, connected, 
quasi-projective, complex algebraic varieties? 
Following Catanese \cite{Cat03}, 
we shall call a group $G$ arising in this fashion a 
{\em quasi-projective}\/ group.

In this context, one may also consider the  
larger class of quasi-compact K\"{a}hler manifolds, 
of the form $M= \oM \setminus D$, where 
$\oM$ is compact K\"{a}hler and $D$ is a normal 
crossing divisor.  If $G=\pi_1(M)$ with $M$ as above, 
we say $G$ is a {\em quasi-K\"{a}hler} group.

The first obstruction to quasi-projectivity was given 
by J. Morgan:   If $G$ is a quasi-projective group, then 
$E_G=\widehat{\bL}/J$, where $\bL$ is a free 
Lie algebra with generators in degrees $1$ and $2$,  
and $J$ is a closed Lie ideal, generated in degrees 
$2$, $3$ and $4$; see \cite[Corollary 10.3]{M}.  
By refining Morgan's techniques, Kapovich and Millson 
obtained analogous quasi-homogeneity restrictions, 
on certain singularities of representation varieties of $G$ 
into reductive algebraic Lie groups; see \cite[Theorem 1.13]{KM}.  
Another obstruction is due to Arapura:  If $G$ is  
quasi-K\"{a}hler, then the characteristic variety $\V_1(G)$ 
is a union of (possibly translated) subtori of $\T_G$; 
see \cite[p.~564]{A}. 

If the group $G$ is $1$-formal, then $E_G=\widehat{\bL}/J$, 
with $\bL$ generated in degree $1$ and $J$ generated in 
degree $2$; thus, $G$ verifies Morgan's test.  It is therefore 
natural to explore the relationship between $1$-formality 
and quasi-projectivity.  (In contrast with the K\"{a}hler case, 
it is known from \cite{M, KM} that these two properties are 
independent.) Another motivation for our investigation comes 
from the study of fundamental groups of complements of plane 
algebraic curves. This class of $1$-formal, quasi-projective 
groups contains, among others, the celebrated Stallings group;  
see \cite{PS06}.

\subsection{Position and resonance obstructions}
\label{subsec=resobs}  

Our main contribution to Serre's problem is Theorem 
\ref{thm=posobstr} below, which provides a new type of 
restriction on fundamental groups of smooth, quasi-projective 
complex algebraic varieties. In the presence of $1$-formality, 
this restriction is expressed entirely in terms of a 
classical invariant, namely the cup-product map 
$\cup _G\colon \bigwedge ^2H^1(G,\C) \to H^2(G,\C)$. 

\begin{thm} 
\label{thm=posobstr} 
Let $M$ be a  connected, quasi-compact K\"{a}hler manifold. 
Set $G=\pi_1(M)$ and let $\{ \V^{\alpha} \}$ be the 
irreducible components of $\V_1(G)$ containing $1$. Denote 
by $\TT^{\alpha}$ the tangent space at $1$ of $\V^{\alpha}$.
Then the following hold.

\begin{enumerate}
\item  \label{a1}  
Every positive-dimensional tangent space $\TT^{\alpha}$ 
is a $p$-isotropic linear subspace of $H^1(G, \C)$, 
of dimension at least $2p+2$, for some 
$p=p(\alpha) \in \{0,1\}$. 

\item  \label{a2} 
If $\alpha \ne \beta $, then $\TT^{\alpha} \bigcap \TT^{\beta}=\{0\}$.
\end{enumerate}
Assume further that $G$ is $1$-formal.  Let $\{ \R^{\alpha} \}$ 
be the irreducible components of $\R_1(G)$.  Then 
the following hold.
\begin{enumerate}
\setcounter{enumi}{2}
\item \label{a0}
The collection $\{ \TT^{\alpha}\}$ coincides with the 
collection $\{ \R^{\alpha}\}$.

\item  \label{a3}
For $1\leq k\leq b_1(G)$, we have 
$\R_k(G)=\{0\} \cup \bigcup\nolimits_\alpha \R^{\alpha}$, 
where the union is over all components  $\R^{\alpha}$ such that 
$\dim \R^{\alpha} >k+p(\alpha)$.  

\item \label{a4}
The group $G$ has a free quotient of rank at least $2$ 
if and only if $\R_1(G)$ strictly contains $\{ 0\}$.
\end{enumerate}

\end{thm}

Here, we say that a non-zero subspace $V\subseteq  H^1(G, \C)$ 
is $0$- (respectively, $1$-) {\em isotropic}\/ if the restriction 
of the cup-product map, 
$\cup_G\colon \bigwedge^2 V\to \cup_G (\bigwedge^2 V)$, 
is equivalent to 
$\cup_C \colon \bigwedge^2 H^1(C, \C)\to H^2(C, \C)$,
where $C$ is a non-compact (respectively, compact) smooth, 
connected complex curve. See \ref{def=position} for a more 
precise definition. The relation between Theorem \ref{thm=posobstr} 
and the isotropic subspace theorems of Bauer \cite{Bau} and 
Catanese \cite{Cat} is discussed in detail in \cite{D08}.

In this paper, we consider only components of $\V_1(G)$ 
containing the origin. For a detailed analysis of translated 
components, we refer to \cite{D2} and \cite{DPS-codone}. 

Theorem \ref{thm=posobstr} is derived from basic results of 
Arapura \cite{A} on quasi-K\"{a}hler groups, plus our 
Theorem \ref{thm=bnew}, in the $1$-formal case.  
Part \eqref{a2} is a novel viewpoint, further pursued in 
\cite{DPS-codone} to obtain a completely new type 
of obstruction to quasi-projectivity. Part \eqref{a0} follows from 
$1$-formality alone, cf.~Theorem \ref{thm=bnew}\eqref{bn3}. 
As shown in Examples \ref{ex:nonorientable} and \ref{ex=2134}, 
the other four properties from Theorem \ref{thm=posobstr} require 
in an essential way the quasi-K\"{a}hler hypothesis.

For an arrangement complement $M$, Parts 
\eqref{a2} and \eqref{a3} of the above theorem were proved 
by Libgober and Yuzvinsky in \cite{LY}, using completely 
different methods. 

The ``position" obstructions \eqref{a1} and \eqref{a2} in
Theorem \ref{thm=posobstr} can be viewed as a 
much strengthened form of Arapura's theorem:
they give information on 
how the components of $\V_1(G)$ passing through 
the origin intersect at that point, and how their tangent 
spaces at $1$ are situated with respect to the 
cup-product map $\cup_G$.   

We may also
view conditions \eqref{a1}--\eqref{a4} as a set of ``resonance" 
obstructions for a $1$-formal group to be quasi-K\"{a}hler, 
or for a quasi-K\"{a}hler group to be $1$-formal.  
Since the class of homotopy types of compact K\"{a}hler 
manifolds is strictly larger than the class of homotopy types 
of smooth projective varieties (see Voisin \cite{V}), one 
may wonder whether the class of quasi-K\"{a}hler 
groups is also strictly larger than the class of 
quasi-projective groups.

\subsection{Applications}
\label{subsec=apps}
In the last four sections, we illustrate the efficiency of our  
obstructions to $1$-formality and quasi-K\"{a}hlerianity with 
several classes of examples.  

In Section \ref{sec=wp}, we consider wedges and products 
of spaces. Our analysis of resonance varieties of wedges,  
together with Theorem \ref{thm=posobstr},  shows that 
$1$-formality and quasi-K\"{a}hlerianity behave quite differently 
with respect to the coproduct operation for groups. 
Indeed, if $G_1$ and $G_2$ are $1$-formal, 
then $G_1*G_2$ is also $1$-formal; but, if in addition 
both factors are non-trivial, presented by 
commutator relators only, and one of them is non-free,
then $G_1*G_2$ is not quasi-K\"{a}hler. 
As a consequence of the position obstruction from Theorem
\ref{thm=posobstr}\eqref{a1}, we also show that the
quasi-K\"{a}hlerianity of $G_1*G_2$, where the groups $G_i$
are assumed only finitely presented with infinite abelianization,
implies the vanishing of $\cup_{G_1}$ and $\cup_{G_2}$.

When it comes to resonance varieties, real subspace 
arrangements offer a stark contrast to complex 
hyperplane arrangements, cf.~\cite{MS1, MS2}.  
If $M$ is the complement of an arrangement of transverse 
planes through the origin of $\RR^4$, then $G=\pi_1(M)$ 
passes Sullivan's $\gr$-test.  Yet, as we note in 
Section \ref{sec=realarr}, the group $G$ may fail 
the tangent cone formula from Theorem \ref{thm=tcfintro}, 
and thus be non-$1$-formal; or, $G$ may be $1$-formal, 
but fail tests \eqref{a1}, \eqref{a2}, \eqref{a3} from 
Theorem \ref{thm=posobstr}, and thus be 
non-quasi-K\"{a}hler.

In Section \ref{sec=conf}, we apply our techniques to the 
configuration spaces $M_{g,n}$ of $n$ ordered points 
on a closed Riemann surface of genus $g$.  Clearly, 
the surface pure braid groups $P_{g,n}=\pi_1(M_{g,n})$ 
are quasi-projective.  On the other hand, if $n\ge 3$, 
the variety $\R_1(P_{1,n})$ is irreducible and non-linear. 
Theorem \ref{thm=posobstr}\eqref{a0} shows that 
$P_{1,n}$ is not $1$-formal, thereby verifying a 
statement of Bezrukavnikov \cite{B}.  

We conclude in Section \ref{sec=artinalg} with a study 
of Artin groups associated to finite, labeled graphs, 
from the perspective of their cohomology jumping loci.  
As shown in \cite{KM}, all Artin groups are $1$-formal; 
thus, they satisfy Morgan's homogeneity test. Moreover, 
as we show in Proposition \ref{prop=vart} (building on 
work from \cite{PS1}), the first characteristic varieties  
of right-angled Artin groups are unions of coordinate 
subtori; thus, all such groups pass Arapura's $\V_1$-test.

In \cite[Theorem 1.1]{KM}, Kapovich and Millson establish, 
by a fairly involved argument, the existence of infinitely many 
Artin groups that are not realizable by smooth, quasi-projective 
varieties. Using the isotropicity obstruction from Theorem 
\ref{thm=posobstr}\eqref{a1}, we show that a right-angled Artin group 
$G_{\Gamma}$ is quasi-K\"{a}hler  if and only if $\Gamma$ 
is a complete, multi-partite graph, in which case $G_{\Gamma}$ 
is actually quasi-projective.  This result provides a complete---and 
quite satisfying---answer to Serre's problem within this class of 
groups.  In the process, we take a first step towards 
solving the problem for all Artin groups, by answering it at 
the level of associated Malcev Lie algebras. We also 
determine precisely which right-angled Artin groups 
are K\"{a}hler.

The approach to Serre's problem taken in this paper---based 
on the obstructions from Theorem \ref{thm=posobstr}---has 
led to complete answers for several other classes of groups:
\begin{itemize}
\item In \cite{DPS-bb}, we classify the quasi-K\"{a}hler 
groups within the class of groups introduced by Bestvina 
and Brady in \cite{BB}.\\[-8pt]
\item In \cite{DPS-codone}, we determine precisely which 
fundamental groups of boundary manifolds of line 
arrangements in $\C\PP^2$ are quasi-projective groups.\\[-8pt]
\item In \cite{DS}, we decide which $3$-manifold groups 
are K\"{a}hler groups, thus answering a question raised by 
S.~Donaldson and W.~Goldman in 1989.\\[-8pt]
\item In \cite{P07}, we show that the fundamental groups of 
a certain natural class of graph links in $\Z$-homology spheres 
are quasi-projective if and only if the corresponding links are 
Seifert links.
\end{itemize}

The computations from Section \ref{subsec=jumpartin} 
have been pursued in \cite{PS08}, leading to a complete 
description of the characteristic varieties, in all degrees, 
for toric complexes associated to arbitrary finite simplicial 
complexes.

The obstructions from Theorem \ref{thm=posobstr} are 
complemented by new methods of constructing interesting 
(quasi-)projective groups.  These methods, developed in 
\cite{DPS-bb} and \cite{DPS-complexbb}, lead to 
a negative answer to the following 
question posed by J.~Koll\'{a}r in \cite{Ko}:  Is every 
projective group commensurable (up to finite 
kernels) with a group admitting a quasi-projective 
variety as classifying space?  

\section{Holonomy Lie algebra, Malcev completion, 
and $1$-formality} 
\label{sec=holo}

Given a  group $G$, there are several 
Lie algebras attached to it:  the associated graded Lie 
algebra $\gr^*(G)$, the holonomy Lie algebra $\h(G)$, 
and the Malcev Lie algebra $E_G$.   
In this section, we review the constructions of these 
Lie algebras, and the related notion of $1$-formality, 
which will be crucial in what follows.  We conclude 
with some relevant facts from the deformation theory 
of representations. 

\subsection{ Holonomy Lie algebras} 
\label{subsecz=holo}
We start by recalling the definition of the holonomy 
Lie algebra, due to K.-T.~Chen \cite{C}. 

Let $M$ be a connected CW-complex.  Let $\K$ be a field of 
characteristic $0$.  Denote by $\bL^*(H_1M)$ the free Lie algebra 
on $H_1M=H_1(M,\K)$, graded by bracket length, and 
use the Lie  bracket to identify $H_1M \wedge H_1M$ 
with  $\bL^2(H_1M)$. Write $\langle U \rangle$ for 
the Lie ideal spanned by a subset $U\subset \bL^*(H_1M)$. 
Set
\begin{equation} 
\label{eq=holm}
\h^* (M):=\bL^*(H_1M)/\langle \im (\partial_M\colon 
H_2M  \to  \bL^2(H_1M))\rangle , 
\end{equation}
where $\partial_M$ is induced by the homology diagonal, 
$H_2M \to H_2(M\times M)$, via the K\"{u}nneth decomposition. 
When $M$ has finite $1$-skeleton, the dual of $\partial_M$ is
the cup-product map 
$\cup_M\colon H^1(M,\K) \wedge H^1(M,\K) \to H^2(M,\K)$. 
Note that $\h (M)$ is a quadratic Lie algebra, in that 
it has a presentation with generators in degree $1$ and 
relations in degree $2$ only. We call $\h (M)$ the 
{\em holonomy Lie algebra}\/ of $M$ (over the field $\K$).

Now let $G$ be a group, with Eilenberg-MacLane space 
$K(G,1)$. Define
\begin{equation} 
\label{eq=holg}
\h (G):=\h (K(G,1)) .
\end{equation}
If $M$ is a CW-complex  with $G=\pi_1(M)$, and 
if $f\colon M \to K(G,1)$ is a classifying map, then $f$ induces 
an isomorphism on $H_1$ and an epimorphism on $H_2$. 
This implies that
\begin{equation} 
\label{eq=holeq}
\h (G)=\h (M) .
\end{equation}

\subsection{Malcev Lie algebras} 
\label{subsec=malcev}
Next, we recall some useful facts about the functorial Malcev 
completion of a group, following Quillen \cite[Appendix A]{Q}.  

A {\em Malcev Lie algebra}\/ is a $\K$-Lie algebra $E$, endowed 
with a decreasing, complete filtration
\[
E=F_1E \supset \cdots \supset F_qE \supset  F_{q+1}E \supset \cdots ,
\]
by $\K$-vector subspaces  satisfying $[F_rE,F_sE] \subset  F_{r+s}E$ 
for all $r,s \ge 1$, and with the property that the associated graded 
Lie algebra, $\gr_F^*(E)=\bigoplus_{q\ge 1} F_qE /F_{q+1}E$, is 
generated by $\gr_F^1(E)$.  By completeness of the filtration, the 
Campbell-Hausdorff formula from local Lie theory
\begin{equation} 
\label{eq=ch}
e \cdot f=e+f+\tfrac{1}{2}[e,f]+ 
\tfrac{1}{12}([e,[e,f]]+[f,[f,e]]) + \cdots 
\end{equation}
endows the underlying set of $E$ with 
a group structure,  denoted by $\eexp(E)$. Clearly, $\eexp(E)$ is
a uniquely divisible group.

The lower central series of a Lie algebra $L$ is defined inductively 
by $\Gamma_1 L=L$ and $\Gamma_{q+1} L= [L, \Gamma_q L]$. If $L$ is 
nilpotent, the lower central series filtration gives a canonical 
Malcev Lie structure on $L$.

For a group $G$, denote by $\wh (G)$ the completion of the 
holonomy Lie algebra with respect to the degree filtration.  
It is readily checked that $\wh (G)$, together with the canonical 
filtration of the completion, is a Malcev Lie algebra with 
$\gr_F^*(\wh (G))=\h^* (G)$.

In \cite{Q}, Quillen associates to $G$, in a functorial 
way, a Malcev Lie algebra $E_G$ and a group homomorphism 
$\kk _G\colon G \to \eexp(E_G)$. A  key property of the Malcev 
completion $(E_G,\kappa_G)$ is that  $\kk _G$ induces an 
isomorphism of graded $\K$-Lie algebras
\begin{equation} 
\label{eq=griso}
\gr^*( \kk _G)\colon \gr^*(G) \otimes \K  
\xrightarrow{\,\simeq\,} \gr^*_F(E_G) .
\end{equation}
Here, $\gr^*(G)=\bigoplus_{q\ge 1} \Gamma_qG /\Gamma_{q+1} G$ 
is the graded Lie algebra associated to the lower central series,
$G =\Gamma _1G \supset \cdots \supset \Gamma _qG  
\supset \Gamma _{q+1}G \supset \cdots$, defined inductively 
by setting $ \Gamma _{q+1}G=(G, \Gamma _qG)$, where 
$(A,B)$ denotes the subgroup generated by all commutators 
$(g,h)=ghg^{-1}h^{-1}$ with $g\in A$ and $h\in B$, and with the 
Lie bracket on $\gr^*(G)$ induced by the commutator map 
$(\:,\:)\colon G\times G\to G$. Note that the Lie algebra $\gr (G)$ 
is generated by $G_{\ab}:= \gr^1(G)$, the abelianization of $G$.

If $N$ is a nilpotent group, $E_N$ is a nilpotent Lie algebra, 
and the Malcev filtration coincides with the lower central series 
filtration. Hence, $\eexp(E_N)$ is a nilpotent group. The Malcev 
completion $\kappa_N \colon N \to \eexp(E_N)$ is universal for 
homomorphisms of $N$ into nilpotent, uniquely divisible groups, 
and has torsion kernel and cokernel.

\subsection{Formal spaces and $1$-formal groups} 
\label{subsec=formal}

The important notion of formality of a space was introduced by 
D.~Sullivan in \cite{S}.  Let $M$ be a connected CW-complex. Roughly 
speaking, $M$ is formal if the rational cohomology algebra of $M$ 
determines the rational homotopy type of $M$.  
Many interesting spaces are formal: compact K\"{a}hler 
manifolds \cite{DGMS}, homogeneous spaces of compact 
connected Lie groups with equal ranks \cite{S}; 
products and wedges of formal spaces are again formal.

\begin{definition} 
\label{def=1formal} 
A finitely generated group $G$ is {\em $1$-formal}\/  
if its Malcev Lie algebra, $E_G$, is isomorphic to the 
completion of its holonomy Lie algebra, $\wh (G)$, 
as filtered Lie algebras.
\end{definition}

A fundamental result of Sullivan \cite{S} states that 
$\pi _1(M)$ is $1$-formal whenever $M$ is formal, with 
finite $1$-skeleton.  The converse is not true in general, 
though it holds under certain additional assumptions, 
see \cite{PS1}. It is well-known that $G$ is $1$-formal 
if and only if $E_G$ is isomorphic to the degree completion 
of a quadratic Lie algebra, as filtered Lie algebras; see for 
instance \cite{ABC}. Here are some motivating examples. 

\begin{example} 
\label{ex:w2formal}
If $M$ is obtained from a smooth, complex projective 
variety $\oM$ by deleting a subvariety 
$D \subset \oM$ with $\codim D \geq 2$, 
then $\pi_1(M)=\pi_1(\oM)$. Hence, $\pi_1(M)$ 
is $1$-formal, by \cite{DGMS}.
\end{example}

\begin{example} 
\label{ex:w1formal}
Let $W_*(H^m(M,\C))$ be the Deligne weight filtration \cite{D} 
on the cohomology with complex coefficients of a connected 
smooth quasi-projective variety $M$. It follows from a basic 
result of J.~Morgan \cite[Corollary 10.3]{M} that 
$\pi _1(M)$ is $1$-formal if $W_1(H^1(M,\C))=0$. 
By \cite[Corollary 3.2.17]{D}, this vanishing property 
holds whenever $M$ admits a non-singular compactification 
with trivial first Betti number. As noted in \cite{K}, these 
two facts together establish the $1$-formality of fundamental 
groups of complements of projective hypersurfaces. 
\end{example}

\begin{example} 
\label{ex:w3formal}
If $b_1 (G)=0$, it follows from \cite{Q} that 
$E_G \cong \wh (G)=0$, therefore $G$ is $1$-formal. 
If $G$ is finite, Serre showed in \cite{Se} that $G$ is 
the fundamental group of a smooth complex projective variety.
\end{example}

\subsection{Deformation theory of representations}
\label{ss24}

We close this section by recalling from Goldman--Millson \cite{GM}
several relevant facts in deformation theory, together with an 
application from Kapovich--Millson \cite{KM} to $1$-formal groups.

The key point here is the description of the category of germs of 
$\C$-analytic varieties, $(X,x)$, by functors of Artin rings. Let 
$(\widehat{\OO}, \m)$ denote the complete local ring of $(X,x)$. 
The corresponding functor associates to an Artin 
local $\C$-algebra $(A, \m)$ the set of $A$-points, 
$F_{X,x}(A)= \Hom_{\loc}(\widehat{\OO}, A)$, where 
$\Hom_{\loc}$ denotes local $\C$-algebra morphisms.

Let $H$ be a $\C$-algebra and $L$ a $\C$-Lie algebra. 
On the vector space $H\otimes L$, define a bracket by
\begin{equation}
\label{eq=algdef}
[a\otimes x, b\otimes y]= ab\otimes [x,y],\ 
\text{for $a,b\in H$ and $x,y\in L$}.
\end{equation}
If $H$ is a commutative algebra, this functorial construction 
produces a Lie algebra $H\otimes L$, whereas if  
$H^{\bullet}$ is graded-commutative, it produces 
a graded Lie algebra $H^{\bullet}\otimes L$ 
(for which the Lie identities hold up to the standard 
sign conventions).  

Now let $G$ be a  discrete, finitely generated group, and $\BB$ 
a linear algebraic group, with Lie algebra $\bb$. Denote 
by $\Rep (G, \BB)$ the variety of group representations 
of $G$ into $\BB$, and by $\Rep (\h(G), \bb)$ the variety of  
Lie algebra homomorphisms from $\h(G)$ to $\bb$.

Denote by $(X,x)$ the analytic germ of $\Rep (G, \BB)$ at 
the trivial representation $1$, and by $F=F_{X,x}$ 
the corresponding functor.  Given an Artin ring $(A, \m)$, 
the Lie algebra $\m\otimes \bb$ is nilpotent, 
since $\m^q=0$ for $q\gg 0$. 
It follows from \cite[Theorem 4.3]{GM} that 
\begin{equation}
\label{eq:fagmb}
F(A) = \Hom_{\groups}(G, \eexp(\m \otimes \bb)).
\end{equation}

Using this description, together with the universality property 
of the Malcev completion, Kapovich and Millson obtained in 
\cite[Theorem 17.1]{KM} the following result. 

\begin{theorem}[\cite{KM}] 
\label{thm=kms}
If the group $G$ is $1$-formal, there is an analytic isomorphism, 
\[
(\Rep (\h(G), \bb), 0)\isom (\Rep (G, \BB), 1)\, ,
\]
natural in $\BB$.
\end{theorem}

\begin{remark}
\label{rem=bab}
Suppose $\BB$ is abelian, in which case $\bb$ is also abelian. 
For a finitely generated group $G$, denote by 
$G_{\abf}= G_{\ab}/\tors (G_{\ab})$ its maximal torsion-free 
abelian quotient, and by $n=\rank (G_{\abf})$ its first Betti number. 
We then have natural analytic isomorphisms, 
\begin{align*}
&(\Rep (G, \BB), 1)\cong (\Rep (G_{\abf}, \BB), 1)\cong (\BB, 1)^n,\\
&(\Rep (\h(G), \bb), 0)\cong (\Rep (\h(G_{\abf}), \bb), 0)\cong (\bb, 0)^n.
\end{align*}
It follows that the analytic isomorphism from Theorem \ref{thm=kms} 
is induced by the local isomorphism $\exp \colon (\bb, 0)\isom (\BB, 1)$. 
When $\BB= \GL_1(\C)=\C^*$, the map 
$\exp \colon (\C, 0)\to (\C^*, 1)$ is just the usual exponential.
\end{remark}

\begin{remark}
\label{rem=bgl}
When $\BB= \GL_r(\C)$, the usual matrix exponential, 
$\exp \colon \m \otimes \gl_r(\C) \to \GL_r(A)$, induces a group isomorphism,
\begin{equation}
\label{eq:bgl}
\exp \colon \eexp(\m \otimes \gl_r(\C))\isom 1+ \m \otimes \gl_r(\C),
\end{equation}
where the right-hand side denotes the subgroup of $\GL_r(A)$ 
consisting of matrices of the form $\id +X$, with all entries of $X$ 
belonging to $\m$.
\end{remark}

\section{Germs of cohomology jump loci}
\label{sec=reljump}

In this section we prove Theorem \ref{thm=tcfintro} from the Introduction.

\subsection{Relative characteristic and resonance varieties}
\label{ss31}
Let $M$ be a connected CW-complex with finite $1$-skeleton. 
Set $G=\pi_1(M)$. Let $\BB$ be a complex linear algebraic group, 
and fix a rational representation, $\rho\colon \BB \to \GL(V)$, 
where $V$ is a complex vector space. Denote by 
$\theta\colon \bb \to \gl(V)$ the morphism of Lie algebras  
induced by $\rho$ on tangent spaces. Recall the relative 
characteristic varieties $\V_k^i(M, \rho)$ defined in 
\eqref{eq:charvarx}, and set 
\begin{equation}
\label{eq=relv}
\V_k^i(G, \rho):= \V_k^i(K(G, 1), \rho). 
\end{equation}
By considering a classifying map $M\to K(G,1)$, it is 
readily seen that $\V_k^1(M, \rho)=\V_k^1(G, \rho)$. 

In general, the (germs of) relative characteristic varieties 
contain more information than the usual characteristic varieties, 
$\V_k^i(G)=\V_k^i(G, \id_{\C^*})$. For instance, if $G$ is 
a finitely generated nilpotent group, then $\V_k^i(G)$ 
depends only on $k$ and the usual Betti number $b_i(G)$; 
see \cite{MacP}.  On the other hand, $\V_k^i(G, \rho)$ 
may depend on more subtle data.  We illustrate this point---and 
the usefulness of the relative characteristic varieties---with 
a simple example, involving nilpotent groups. 

\begin{example}
\label{ex:heisenberg group}
Let $G=\langle x_1, x_2 \mid ((x_1, x_2),x_1),\, ((x_1, x_2),x_2) \rangle$ 
be the Heisenberg group, with abelianization $G_{\ab}=\Z^2$. 
It is readily seen that $\V^1_k(G)=\V^1_k(\Z^2)$: both are equal to 
$\{1\}$ if $k\le 2$, and are empty otherwise.  

Now let $\rho \colon \BB=\SL(2,\C)\hookrightarrow \GL(2,\C)$ be 
the inclusion map. It is not difficult to check that both (reduced) germs,
$(\Rep (G, \BB), 1)$ and $(\Rep (\Z^2, \BB), 1)$, are isomorphic to 
the germ at $0$ of the variety of commuting pairs of matrices from
$\bb =\sl (2, \C)$; see \cite[p.~89]{GM}.

Direct computation shows that $\V^1_3(\Z^2,\rho)= \{1\}$. 
On the other hand, consider the embedding
$\iota\colon (\C^2, 0)\hookrightarrow (\Rep (G, \BB), 1)$, which, 
under the above isomorphism, sends $(u,v)$ to the pair 
$\big( \big(\begin{smallmatrix} 0 & u \\ 0 & 0 
\end{smallmatrix}\big) , \big(\begin{smallmatrix} 0 & v \\ 0 & 0 
\end{smallmatrix}\big) \big)\in \bb^2$. Another direct 
computation (compare with \cite{Al}) shows that 
$\iota((\C^2, 0))\subset (\V^1_3(G,\rho) ,1)$. Hence, 
$(\V^1_3(G,\rho) ,1)\not\cong (\V^1_3(\Z^2,\rho) ,1)$.
\end{example}

Now let $\h$ be a complex, finitely generated Lie algebra, and 
let $\theta \colon \bb \to \gl(V)$ be a complex, finite-dimensional 
representation of a complex, finite-dimensional Lie algebra $\bb$. 
Define the {\em relative resonance varieties}\/ of $\h$ (with respect 
to $\theta$) by
\begin{equation}
\label{eq=relr}
\R_k^i(\h ,\theta)= \{ \theta'\in \Rep( \h, \bb) \mid 
\dim_{\C} H^i(\h, {}_{\theta\theta'} V)\ge k \},
\end{equation}
where Lie algebra cohomology is taken by viewing $V$ as a left $\h$-module 
via $\theta\theta'$.  For each $i$, we have a descending filtration 
$\{ \R_k^i(\h, \theta)\}_{k\ge 0}$ on $\Rep(\h, \bb)$. Note  
that the sets $\R_k(\h, \theta):=\R_k^1(\h, \theta)$ 
are closed subvarieties of $\Rep(\h, \bb)$. 
As we shall see in Corollary \ref{cor=2res}, 
if $\h=\h(G)$ is the holonomy Lie algebra of $G$, then 
$\R_k(\h, \theta)=\R_k(G, \theta)$.

As before, write $\R_k(\h) := \R_k^1(\h, \id_{\gl_1(\C)})$. 
If $\h=\h(G)$, this is a subvariety of $\Rep(\h, \gl_1(\C))=H^1(G, \C)$. 

\subsection{Low-dimensional cohomology and representations}
\label{ss32}

We first recall some standard facts from \cite[Chapters VI and VII]{HS}.

Let $G$ be a finitely generated group and $\rho\colon \BB \to \GL(V)$ 
a rational representation.  Denote by $Z^i(G, \cdot)$ the $i$-cocycles 
for group cohomology, for $i=0,1$. Consider the semidirect 
product of groups $V\rtimes_{\rho}\BB$, and the morphism of varieties
\begin{equation*}
\label{eq=z1gr}
p\colon \Rep(G, V\rtimes_{\rho}\BB) \rightarrow \Rep(G, \BB)
\end{equation*}
induced by the canonical projection, $V\rtimes_{\rho}\BB \to \BB$.

\begin{lemma}
\label{lem=z1gr}
With notation as above, the fiber of $p$  
over $\rho'$ is $Z^1(G, {}_{\rho\rho'}V)$.
\end{lemma}

Consider the variety 
\[
\Rep_{\rho}(G, \BB)= \{ (v, \rho')\in V\times \Rep(G, \BB)\mid 
\rho\rho'(g)\cdot v=v,\ \text{for all $g\in G$}\},
\] 
and the morphism induced by the second-coordinate projection,
\begin{equation*}
\label{eq=z0gr}
q\colon \Rep_{\rho}(G, \BB)\rightarrow \Rep(G, \BB) .
\end{equation*}

\begin{lemma}
\label{lem=z0gr}
The fiber of $q$ over $\rho'$ is $Z^0(G, {}_{\rho\rho'}V)$.
\end{lemma}

Let $\h$ be a finitely generated Lie algebra and $\theta \colon\bb \to \gl(V)$ 
a representation, with the properties described at the end of \S\ref{ss31}. 
Consider the semidirect product  Lie algebra, $V\rtimes_{\theta}\bb$, 
and the morphism of varieties
\begin{equation*}
\label{eq=z1lie}
\mathsf{p}\colon \Rep(\h, V\rtimes_{\theta}\bb)\rightarrow \Rep(\h, \bb)
\end{equation*}
induced by the canonical projection, $V\rtimes_{\theta}\bb \to \bb$.

\begin{lemma}
\label{lem=z1lie}
The fiber of $\mathsf{p}$ over $\theta'$ is $Z^1(\h, {}_{\theta\theta'}V)$.
\end{lemma}

Now consider the variety 
\[
\Rep_{\theta}(\h, \bb)= \{ (v, \theta')\in V\times \Rep(\h, \bb)\mid 
\theta\theta'(h)\cdot v=0, \ \text{for all $h\in \h$} \},
\] 
and the morphism induced by the second-coordinate projection,
\begin{equation*}
\label{eq=z0lie}
\mathsf{q}\colon \Rep_{\theta}(\h, \bb)\rightarrow \Rep(\h, \bb).
\end{equation*}

\begin{lemma}
\label{lem=z0lie}
The fiber of $\mathsf{q}$ over $\theta'$ is 
$Z^0(\h, {}_{\theta\theta'}V)$.
\end{lemma}

\begin{remark}
\label{rem=tgsemidir}
If $\rho\colon \BB\to \GL(V)$ is a rational representation, 
and $\theta=d_1(\rho)\colon \bb \to \gl(V)$ is the induced  
representation on tangent spaces, then the Lie algebra of 
the algebraic group $V\rtimes_{\rho}\BB$ is $V\rtimes_{\theta}\bb$; 
see for instance \cite{Hu}.
\end{remark}

\subsection{Beginning of proof of Theorem \ref{thm=tcfintro}}
\label{ss33}

We need one more lemma.

\begin{lemma}
\label{lem=z0art}
Let $(A, \m)$ be an Artin local algebra. Given $v\in A^r$ and 
$s\in \m\otimes \gl_r(\C)$, the following are equivalent:
\begin{enumerate}
\item \label{zz1}
$\eexp(s)^k \cdot v=v$, for some $k\in \Z\setminus \{ 0\}$.
\item \label{zz2}
$\eexp(s) \cdot v=v$.
\item \label{zz3}
$s\cdot v=0$.
\end{enumerate}
\end{lemma}

\begin{proof}
The implications \eqref{zz2}$\Rightarrow$\eqref{zz1} and 
\eqref{zz3}$\Rightarrow$\eqref{zz2} are clear.  

Now suppose \eqref{zz2} holds. Then $as\cdot v=0$, for some 
$a\in 1+ \m\otimes \gl_r(\C)\subseteq \GL_r(A)$. Hence, 
\eqref{zz3} holds. 

Finally, \eqref{zz1}$\Rightarrow$\eqref{zz3} follows from
\eqref{zz2}$\Rightarrow$\eqref{zz3}, applied to $ks$.
\end{proof}

\begin{theorem}
\label{thm=a1}
Let $G$ be a $1$-formal group, and $\rho\colon \BB \to \GL(V)$ a 
rational representation.  Set $\theta= d_1(\rho)\colon \bb\to \gl(V)$. 
Then, for each $k\ge 1$, the analytic isomorphism from 
Theorem \ref{thm=kms} induces an analytic isomorphism 
of germs,
\[
(\R_k(\h(G), \theta), 0) \isom (\V_k(G, \rho), 1).
\] 
\end{theorem}

\begin{proof}
Assume in Theorem \ref{thm=kms} that $\theta'\in \Rep(\h(G), \bb)$ 
corresponds to $\rho'\in \Rep(G, \BB)$. We need to show that 
$\dim_{\C} Z^i(G, {}_{\rho\rho'}V)=\dim_{\C} Z^i(\h(G), {}_{\theta\theta'}V)$, 
for $i=0,1$. For $i=1$, this follows from Lemmas \ref{lem=z1gr} 
and \ref{lem=z1lie}, by using the naturality property from 
Theorem \ref{thm=kms}.

Denote by $F_X$ the functor of Artin rings corresponding to the 
germ at $(0,1)$ of the variety $\Rep_{\rho}(G, \BB)$.  Set 
$r=\dim_{\C}V$. It follows from \S\ref{ss24} that 
\[
F_X(A)= \{ (v, \rho'')\in \m^{\times r}\times 
\Rep(G, \eexp(\m\otimes \bb))\mid 
(\id\otimes \theta)\rho''(g)\cdot v=v, \: \forall g\in G\},
\]
where $\id\otimes \theta\colon \eexp(\m\otimes \bb)\to 
\eexp(\m\otimes \gl_r(\C))$ is the group morphism 
induced by $\theta$. Recall from \eqref{eq:bgl}
that $\eexp(\m\otimes \gl_r(\C))\cong 1+ \m \otimes \gl_r(\C)$. 
Using Lemma \ref{lem=z0art},  the universality 
and torsion properties of the Malcev completion explained 
in \S\ref{subsec=malcev}, and the $1$-formality  
of $G$, we may naturally identify $F_X(A)$ with 
\[
\{ (v, \theta'')\in \m^{\times r}\times \Rep(\h(G), \m\otimes \bb)\mid 
(\id\otimes \theta)\theta''(h)\cdot v=0,\: \forall h\in \h(G) \}.
\]

Plainly, this set coincides with $F_Y(A)$, where $F_Y$ is the functor 
of Artin rings corresponding to the germ at $(0,0)$ of the variety 
$\Rep_{\theta}(\h(G), \bb)$.  By Lemmas \ref{lem=z0gr} and \ref{lem=z0lie}, 
$\dim_{\C} Z^0(G, {}_{\rho\rho'}V)=\dim_{\C} Z^0(\h(G), {}_{\theta\theta'}V)$, 
and we are done.
\end{proof}

\subsection{Resonance and Aomoto complexes}
\label{ss34}

We now relate Lie algebra representations to quadratic cones. 

Let $M$ be a connected CW-complex with finite $1$-skeleton and
fundamental group $G=\pi_1(M)$. Denote by $H^{\bullet} M$ the 
cohomology ring of $M$ with $\C$ coefficients. Given a complex 
Lie algebra $L$, consider the graded Lie algebra
$H^{\bullet} M\otimes L$ constructed in \eqref{eq=algdef}. 
Define the associated {\em quadratic cone}\/ by 
\begin{equation}
\label{eq=qcone}
\QQ(M, L)= \{ x\in H^1 M\otimes L\mid [x,x]=0\} .
\end{equation}
Clearly, $\QQ(M, L)$ depends only on $G$, so we denote 
it by $\QQ(G, L)$.

\begin{lemma}
\label{lem=repqc}
Let $G$ be a finitely generated group.
The linear isomorphism between $\Hom_{\C}(\h^1(G), L)$ and 
$H^1G\otimes L$ gives a natural identification,
\[
\Rep(\h(G), L)\cong \QQ(G, L) .
\]
\end{lemma}

\begin{proof}
Plainly, $\Rep(\h(G), L)=\{ f\in \Hom_{\C}(H_1G, L)\mid 
\beta\circ \wedge^2 f\circ \partial_G=0\}$, where 
$\beta\colon L\wedge L\to L$ is the Lie bracket, 
and $\partial_G\colon H_2G\to \bigwedge^2 H_1G$ is 
defined in \eqref{eq=holm}. The element $f$ is naturally 
identified with $x\in H^1G\otimes L$.  We need to check that 
\begin{equation}
\label{eq=repqc}
\beta\circ \wedge^2 f\circ \partial_G=0\ \eqv\ [x,x]=0\, .
\end{equation}

Pick a $\C$-basis $\{ e_i\}$ for $H_1G$, and denote by $\{ e_i^*\}$ 
the dual basis of $H^1G$. Let $\{ y_k\}$ and $\{ z_a\}$ be $\C$-bases 
for $H_2G$ and $L$.  Write $x= \sum_{i,a}t_i^a e_i^*\otimes z_a$, 
where $f(e_i)= \sum_a t_i^a z_a$. Also write 
$\partial_G (y_k)= \frac{1}{2} \sum_{i,j} \mu_{ij}^k e_i\wedge e_j$, 
and $[z_a, z_b]= \sum_{c}s_{ab}^c z_c$. Note that 
$e_i^* \cup e_j^*(y_k)= \mu_{ij}^k$. It is now 
straightforward to check that both conditions from \eqref{eq=repqc} 
are equivalent to $\sum_{i,j,a,b} t_i^a t_j^b \mu_{ij}^k s_{ab}^c=0$, 
for all $k$ and $c$.
\end{proof}

Let $\theta\colon \bb\to \gl(V)$ be a Lie algebra representation 
over $\C$, with $\bb$ and $V$ finite-dimensional. Tensoring 
with $H^{\bullet} M$ the split exact sequence defined by the 
Lie semidirect product $V\rtimes_{\theta}\bb$, we obtain a 
graded vector space decomposition, 
\[
H^{\bullet} M\otimes (V\rtimes_{\theta}\bb)= 
(H^{\bullet} M\otimes V) \oplus (H^{\bullet} M\otimes \bb),
\]
where $H^{\bullet} M\otimes \bb$ is a Lie subalgebra, and 
$H^{\bullet} M\otimes V$ is an abelian Lie ideal.

For $x\in \QQ(M, \bb)\subseteq H^{1}M\otimes (V\rtimes_{\theta}\bb)$, 
we have $0=\ad_{[x,x]}=2 \ad_x^2$.

\begin{definition}
\label{def=relaom}
The {\em relative Aomoto complex} of $H^{\bullet} M$ with respect to 
$\theta\colon \bb\to \gl(V)$ and $x\in \QQ(M, \bb)$ is the subcomplex 
$(H^{\bullet} M\otimes V, \ad_x)$ of the cochain complex
$(H^{\bullet} M\otimes (V\rtimes_{\theta}\bb), \ad_x)$.
\end{definition}

The reason for this terminology is given by the next Lemma. 

\begin{lemma}
\label{lem=abaom}
If $\bb=\C=\gl_1(\C)$ and $\rho=\id_{\bb}$, then every element 
$z\in H^1M\cong H^1M\otimes \bb$ satisfies $[z,z]=0$, 
and the relative Aomoto complex 
$(H^{\bullet} M\otimes V, \ad_{z\otimes 1})$ is identified 
with the usual Aomoto complex, $(H^{\bullet}M, \mu_z)$.
\end{lemma}

\begin{proof}
The first assertion follows at once from \eqref{eq=algdef}, 
since $\bb$ is abelian.  To check the second claim,
pick $x=z\otimes 1\in H^1M\otimes \bb \equiv H^1M$, and 
$y\otimes 1\in H^iM\otimes V\equiv H^iM$.
Then 
\[
[x, y\otimes 1]= zy\otimes [(0,1), (1,0)]= 
z\cup y\otimes 1\in H^{i+1}M\otimes V,
\] 
since the product is computed in the cohomology ring $H^{\bullet}M$, 
and the bracket in the Lie semidirect product $\C\rtimes_{\id} \C$. 
This identifies the relative Aomoto complex 
$(H^{\bullet}M\otimes V, \ad_{x})$ with the usual 
Aomoto complex $(H^{\bullet}M, \mu_z)$, and we 
are done.
\end{proof}

Returning now to the general situation, let 
$G$ be a finitely generated group, and let 
$\theta\colon \bb\to \gl(V)$ be a Lie algebra 
representation as above.  

\begin{lemma}
\label{lem=resaom}
Given $\theta'\in \Rep(\h(G), \bb)$, we have a natural isomorphism
\[
H^1(\h(G), {}_{\theta\theta'}V)\cong H^1(H^{\bullet}G\otimes V, \ad_x),
\]
where $x\in \QQ(G, \bb)$ corresponds to $\theta'$ under the 
isomorphism from Lemma \ref{lem=repqc}.
\end{lemma}

\begin{proof}
Denoting by $Z^1$ and $B^1$ the respective cocycles and 
coboundaries, we will prove the existence of natural compatible 
isomorphisms, for both of them.

For $Z^1$, it follows from Lemmas \ref{lem=z1lie} and \ref{lem=repqc} 
that $Z^1(\h(G), {}_{\theta\theta'}V)$ is naturally identified with 
$\{ a\in H^1G \otimes V\mid [a+x, a+x]=0\}$. Clearly, $[a+x, a+x]=0$ 
if and only if $[x,a]=0$, which happens precisely when 
$a\in Z^1(H^{\bullet}G\otimes V, \ad_x)$.

Under this natural isomorphism, $B^1(\h(G), {}_{\theta\theta'}V)$ 
is identified with the image of the linear map, 
$d\colon V\to \Hom_{\C}(H_1G, V)$, defined by 
$dv(h)= \theta\theta'(h)\cdot v$, for $v\in V$ and 
$h\in H_1G=\h^1(G)$. To finish the proof, we will show that 
$d\equiv \ad_x \colon H^0G\otimes V \to H^1G \otimes V$.

As in the proof of Lemma \ref{lem=repqc}, write 
$x=\sum_i e_i^*\otimes b_i$, with $b_i=\theta'(e_i)$. Then 
\[
[x, 1\otimes v]= \sum_i e_i^*\otimes [b_i, v]= 
\sum_i e_i^*\otimes \theta(b_i)\cdot v=
\sum_i e_i^*\otimes 
\theta\theta'(e_i)\cdot v=dv,
\]
as claimed.
\end{proof}

\begin{remark}
\label{rem=murelres}
One may generalize Definition \ref{def=relaom} to an 
an arbitrary connected, graded-commutative $\C$-algebra
$H^{\bullet}$, by setting $\QQ(H^{\bullet}, \bb)= 
\{ x\in H^1\otimes \bb \mid [x,x]=0 \}$, and proceeding 
as above. If $H^{\bullet} \to H'^{\bullet}$ is a morphism of 
graded $\C$-algebras, which is an isomorphism in 
degree $\bullet =1$ and a monomorphism for 
$\bullet =2$, then the corresponding quadratic cones 
are identified, as well as the cohomology in degree one 
of the associated Aomoto complexes.

Let $H^1$ and $H^2$ be $\C$-vector spaces, 
with $\dim_{\C}H^1< \infty$, and let 
$\mu\colon H^1\wedge H^1\to H^2$ be a linear map.  
Set $H^0=\C\cdot 1$.  Associated to these data, 
there is a connected, graded-commutative algebra, 
$H^{\bullet}(\mu):= H^0\oplus H^1\oplus H^2$.
Define
\[
\R_k(\mu, \theta)= \{ x\in H^1\otimes \bb \mid [x,x]=0 \ \text{and  
$\dim_{\C}H^1(H^{\bullet}(\mu)\otimes V, \ad_x)\ge k$} \}.
\]
Then the varieties $\{ \R_k(\mu, \theta)\}_{k\ge 0}$ depend 
only on the corestriction of $\mu$ to its image, and the 
representation $\theta\colon \bb\to \gl(V)$. As in \S\ref{ss31}, 
set $\R_k(\mu)=\R_k(\mu, \id_{\gl_1(\C)})$.

Now let $M$ be connected complex  with finite $1$-skeleton. 
Note that the degree-$2$ truncation of the cohomology ring, 
$H^{\le 2}(M,\C)$, is isomorphic to $H^{\bullet}(\cup_M)$. 
Thus, the varieties $\R_k(\cup_M,\theta)$ coincide with 
the relative resonance varieties $\R^1_k(M, \theta)$ 
defined in \eqref{eq:resvarx}.  
\end{remark}

\begin{corollary}
\label{cor=2res}
Let $G$ be a  finitely generated group and let 
$\theta\colon \bb \to \gl(V)$ be a morphism of 
complex Lie algebras, with $\bb$ and $V$ finite-dimensional.  
Then the varieties $\R_k(\h(G), \theta)$ and 
$\R_k(G, \theta)$ are naturally isomorphic, for all $k\ge 1$.
\end{corollary}

\subsection{End of proof of Theorem \ref{thm=tcfintro}}
\label{ss35}

The first part follows at once from Theorem \ref{thm=a1} 
and Corollary \ref{cor=2res}.

For basic facts on tangent cones we refer the reader to Whitney's 
book \cite{Wy}.  By Remark \ref{rem=murelres}, we infer that the 
isomorphism from Theorem \ref{thm=a1} induces an isomorphism of 
varieties, $TC_1(\V_k(G, \rho))\cong TC_0(\R_k(\mu_G, \theta))$, 
where $\mu_G$ denotes the corestriction of $\cup_G$ to its image. 
Note that $\dim_{\C}H^{\bullet}(\mu_G)<\infty$, and the (finite) 
matrices of the differential $\ad_x$ of the Aomoto complex 
$(H^{\bullet}(\mu_G)\otimes V, \ad_x)$ have entries consisting of
linear forms. By construction, $\R_k(\mu_G, \theta)$ is a cone, hence 
$TC_0(\R_k(\mu_G, \theta))\cong \R_k(\mu_G, \theta)\cong \R_k(G, \theta)$. 
This proves the second part of Theorem \ref{thm=tcfintro}.

Finally, assume $\BB=\GL_1(\C)=\C^*$ and $\rho=\id_{\BB}$. 
As noted in Remark \ref{rem=bab}, the local isomorphism 
$(\Rep(\h(G), \C), 0)\cong (\Rep(G, \C^*), 1)$ is induced by 
the usual exponential, 
\[
\exp\colon (\Hom_{\groups}(G_{\ab}, \C), 0)
\isom (\Hom_{\groups}(G_{\ab}, \C^*), 1),
\]
which identifies both tangent spaces with $H^1G=H^1(G, \C)$. 
The proof of Theorem \ref{thm=tcfintro} is complete.

\section{$1$-Formality and rationality properties}
\label{sec=linq}

In this section, $G$ is a finitely generated group. We deduce from 
Theorem \ref{thm=tcfintro} two remarkable linearity and rationality 
properties of the resonance varieties $\R_k(G)$, in the presence 
of $1$-formality. 

\subsection{Structure of exponential tangent cones}
\label{ss41}
As usual, let $\T_G=H^1(G, \C^*)=\Hom (G, \C^*)$ be the character 
group of $G$. Consider the exponential homomorphism, 
$\exp\colon H^1(G, \C) \to \T_G$, with image $\T_G^0$, 
the connected component of the identity $1\in \T_G$. 

\begin{definition}
\label{def=tau}
For a Zariski closed subset $W\subseteq \T_G$, define the 
{\em exponential tangent cone}\/ of $W$ at $1$ by
\[
\tau_1(W)= \{ z\in H^1(G, \C) \mid \exp(tz)\in W, \ 
\text{for all $t\in \C$} \}.
\]
\end{definition}

Plainly, $\tau_1(W)$ depends only on the analytic germ of $W$ at $1$.
It is also easy to check that $\tau_1(W)\subseteq TC_1(W)$, whence
our terminology.
The next Lemma is the exponential version of a result of 
Laurent \cite[Lemme 3]{Lau}.

\begin{lemma}
\label{lem=elau}
For any $W$ as above, $\tau_1(W)$ is a finite union of rationally 
defined linear subspaces of $H^1(G, \C)$.
\end{lemma}

\begin{proof}
Set $n=b_1(G)$.  The coordinate ring of the complex torus 
$\T_G^0=(\C^*)^n$ is the Laurent polynomial ring in $n$ 
variables, $\C\Z^n= \C [t_1^{\pm 1}, \dots, t_n^{\pm 1}]$. 
It is enough to verify our claim for $W=V(f)$, where 
$f= \sum_{u\in S} c_u t_1^{u_1}\cdots t_n^{u_n}$, the support 
$S\subseteq \Z^n$ is finite, and $c_u\ne 0$, for all $u\in S$. 
For a fixed $z\in \C^n$, $z\in \tau_1(W)$ if and only if the 
analytic function 
\begin{equation}
\label{eq=sumexp}
\sum_{u\in S} c_u e^{\langle u,z \rangle t}
\end{equation}
vanishes identically in $t$.

In turn, this condition is easy to check, by using the well-known 
fact that the exponential functions $e^{ty_1},\dots , e^{ty_r}$ are linearly 
independent, provided $y_1,\dots, y_r$ are all distinct. Define $\cP$ to be
the set of partitions $p=S_1 \coprod \dots \coprod S_k$ of $S$, 
having the property that $\sum_{u\in S_i} c_u =0$, for $i=1, \dots, k$. 
For $p\in \cP$, define the rational linear subspace $L(p)\subseteq \C^n$ by 
\[
L(p)= \{ z\in \C^n \mid \langle u-v, z\rangle =0, 
\; \forall u,v\in S_i, \; \forall 1\le i\le k \}. 
\]
It is straightforward to check that $\tau_1(W)= \bigcup_{p\in \cP} L(p)$, 
by grouping terms in \eqref{eq=sumexp}.
\end{proof}

\subsection{Proof of Theorem \ref{thm=bnew}}
\label{ss42}
Let $G$ be a (finitely generated) $1$-formal group. 

Part \eqref{bn1}. Theorem \ref{thm=tcfintro} guarantees that 
$\tau_1(\V_k(G))=\R_k(G)$. By Lemma \ref{lem=elau}, all 
irreducible components of $\R_k(G)$ are rationally 
defined linear subspaces of $H^1(G,\C)$. 

Part \eqref{bn2}. Let $\R_k(G)=\bigcup_a L_a$ be the irreducible 
decomposition from Part \eqref{bn1}.  Then $\exp(L_a)$ 
is a connected subtorus of $\T_G$, for all $a$, and 
$T_1(\exp(L_a))= L_a$.  Now let $\{ \V^{\alpha}_k \}_{\alpha}$ 
be the irreducible components of $\V_k(G)$ containing $1$. 
By Theorem \ref{thm=tcfintro}, 
$\bigcup_{\alpha} \V^{\alpha}_k= \bigcup_a \exp(L_a)$, near $1$. 
Hence, each component $\V^{\alpha}_k(G)$ is a connected 
subtorus of $\T_G$. 

Part \eqref{bn3}. Using again Theorem \ref{thm=tcfintro}, 
we deduce that $\R_k(G)= TC_1(\V_k(G))= 
\bigcup_{\alpha} T_1(\V_k^{\alpha})$. 
By \cite[13.1]{Hu}, this gives the decomposition 
into irreducible components of $\R_k(G)$.

\subsection{Formal realizability of cohomology rings}
\label{ss43}

Theorem \ref{thm=bnew} indicates that the $1$-formality property 
of a group imposes severe restrictions on its resonance varieties 
in degree one.

\begin{example}
\label{ex=eucl}
Let $K$ be the finite, $2$-dimensional CW-complex associated to 
the following group presentation, with commutator-relators:
\[
G=\langle x_1, x_2, x_3, x_4 \mid (x_1, x_2)\, , 
(x_1, x_4)\cdot (x_2^{-2}, x_3)\, , 
(x_1^{-1}, x_3)\cdot (x_2, x_4) \rangle\, .
\]
The dual of the cup-product, $\partial_K\colon H_2(K,\C) \to 
\bigwedge^2H_1(K,\C)$, is given by
\begin{equation}
\label{eq=dualcup}
\left \{
\begin{array}{lcl}
\partial_K f_1 & = & e_1\wedge e_2\, ,\\
\partial_K f_2 & = & e_1\wedge e_4 - 2e_2\wedge e_3\, ,\\
\partial_K f_3 & = & -e_1\wedge e_3 + e_2\wedge e_4\, ,
\end{array}
\right.
\end{equation}
where $\{ e_i\}$ are the $1$-cells of $K$, and $\{ f_j\}$ 
the $2$-cells.  It follows from \eqref{eq=dualcup} that 
$\R_1(K)=\R_1(G)$ is given by the equation $x_1^2-2x_2^2=0$, 
where $x=(x_1,x_2,x_3,x_4)\in \C^4$ are the coordinates corresponding 
to the canonical $\Q$-structure of $H^1(K,\C)=H^1(G,\C)$.
\end{example}

Plainly, the rationality property from Theorem \ref{thm=bnew}\eqref{bn1} 
is violated by the resonance variety $\R_1(G)$ from the previous 
example.  This leads to the following corollary. 

\begin{corollary}
\label{cor=notreal}
Let $K$ be the CW-complex defined above.  There is no 
formal CW-complex $M$ with finite $1$-skeleton such that 
$H^{\le 2}(M,\Q)\cong H^{\le 2}(K,\Q)$, as graded rings.  
\end{corollary}

This corollary is in marked contrast with the following basic result in 
simply-connected rational homotopy theory.

\begin{theorem}[Sullivan \cite{S}] 
\label{thm=freal}
Let $H^{\bullet}$ be a connected, finite-dimensional, graded-commu\-tative 
algebra over $\Q$.  If $H^1=0$, there is a $1$-connected, finite, formal 
CW-complex $M$, such that $H^{\bullet}(M, \Q)\cong H^{\bullet}$, 
as graded $\Q$-algebras.
\end{theorem}

\section{Alexander invariants of $1$-formal groups} 
\label{sec=alex}

Our goal in this section is to derive a relation between the 
Alexander invariant and the holonomy Lie algebra of a 
finitely generated, $1$-formal group, over a characteristic zero field $\K$.

\subsection{Alexander invariants} 
\label{subsec=alex} 
 
Let $G$ be a group.  Consider the exact sequence
\begin{equation} 
\label{eq=gsq}
\xymatrix{0 \ar[r]& G'/G''  \ar[r]^{j}& 
G/G''   \ar[r]^{p}&   G_{\ab}  \ar[r]&  0},
\end{equation}
where $G'=(G,G)$, $G''=(G',G')$ and $ G_{\ab}=G/G'$. 
Conjugation in $ G/G''$ naturally makes $G'/G''$ into a 
module over the group ring $\Z{G}_{\ab}$.   Following 
Fox \cite{Fox}, we call this module, 
\[
B_G= G'/G'',
\]
the {\em Alexander invariant}\/ of $G$.  If $G=\pi_1(M)$, where $M$ 
is a connected CW-complex, one has the following useful topological 
interpretation for the Alexander invariant. Let $M' \to M$ be 
the Galois cover corresponding to $G' \subset G$. Then
$B_G  \otimes \K=H_1(M',\K)$, and the action of $G_{\ab}$ 
corresponds to the action in homology of the group of
covering transformations.  

Now assume the group $G$ is finitely generated. Then 
$B_G  \otimes \K$ is a finitely generated module over the 
Noetherian ring $\K G_{\ab}$. Denote by $I \subset \K G_{\ab}$ 
the augmentation ideal and set $X:= G_{\ab} \otimes \K$. 
The $I$-adic completion $\widehat { B_G  \otimes \K}$ is 
a finitely generated module over 
$\widehat {\K G_{\ab}}\cong \K [[X]]$, the formal power series ring 
on $X$. Note that $\K [[X]]$ is also the $(X)$-adic completion 
of the polynomial ring $\K [X]$.

The above identification is induced by the  $\K$-algebra map
\begin{equation} 
\label{eq=e}
\exp\colon \K G_{\ab} \to \K [[X]] , 
\end{equation}
defined  by $\exp (a)=e^{a \otimes 1}$, for $a \in G_{\ab}$.  
After completion, 
we obtain an isomorphism of filtered  $\K$-algebras, 
$\widehat {\exp}\colon \widehat {\K G_{\ab}} 
\xrightarrow{\,\simeq\,} \K [[X]]$.

Another invariant associated to a finitely generated group $G$ 
is the {\em infinitesimal Alexander invariant}, $B_{\h (G)}$.   
By definition, this is the finitely generated graded module over $\K[X]$ 
with presentation matrix
\begin{equation} 
\label{eq=infalex}
\nabla:= \delta _3 + \id \otimes \partial _G\colon 
\K [X] \otimes\Big(\bigwedge\nolimits^3X \oplus Y\Big) \to 
\K [X] \otimes \bigwedge\nolimits^2X ,
\end{equation}
where $Y=H_2(G,\K)$ and $ \delta _3 (x \wedge y  \wedge z)
= x\otimes y   \wedge z -y \otimes x  \wedge z +z 
\otimes x \wedge y$; see \cite[Theorem 6.2]{PS2} for details on 
degrees.  As the notation indicates, the graded module $B_{\h (G)}$ 
depends only on the holonomy Lie algebra of $G$.

\subsection{Alexander invariant and Malcev completion} 
\label{subsec=expchange} 

Let $G$ be a finitely generated group, with Malcev Lie algebra 
$E= E_G$. We have an exact sequence of complete Lie algebras, 
\begin{equation} 
\label{eq=hsq}
\xymatrix{ 0 \ar[r]& \overline{E'}/\overline{E''}  
\ar[r]^{\iota} & E/\overline{E''}
\ar[r]^{\pi} &  E/\overline{E'}\cong X  \ar[r]&  0 } ,
\end{equation}
with $E/\overline{E'}\cong X $ and $\overline{E'}/\overline{E''}$
abelian. Here $\overline{(\cdot)}$ denotes closure with respect 
to the topology defined by the filtration. Both $ E/\overline{E''}$ 
and $E/\overline{E'}$ are endowed with the quotient
filtrations induced from $E$, and  $\overline{E'}/\overline{E''}$ 
carries the subspace filtration induced from $ E/\overline{E''}$. 
The adjoint action of $X$ on $E/\overline{E''}$ induces a 
$\K[[X]]$-module structure on $\overline{E'}/\overline{E''}$.

Theorem~3.5 from \cite{PS2} gives a commutative diagram of groups, 
\begin{equation}
\label{eq=k0diagram}
\xymatrix{
0\ar[r]& G'/G'' \ar^{j}[r] \ar^{\kk_0}[d] &  G/G''\ar^{\kk''}[d]  
\ar[r]^{p}&   G_{\ab}\ar^{\kk'}[d] \ar[r]&0 \\
0\ar[r]&  \eexp(\overline{E'}/\overline{E''})  \ar^{\iota}[r] 
&\eexp(E/\overline{E''}) 
\ar^{\pi}[r] &\eexp(X) \ar[r]&0
}
\end{equation}
where both $\kk'$ and $\kk''$ are Malcev completions.

\begin{lemma}
\label{lem=eq}
The map $\kk _0 \otimes \K \colon (G'/G'')\otimes \K \to 
\overline{E'}/\overline{E''}$ is $\exp$-linear, that is,
\[
(\kk _0 \otimes \K) (\alpha\cdot \beta)= \exp (\alpha)\cdot
(\kk _0 \otimes \K) (\beta)\, ,
\]
for $\alpha \in \K G_{\ab}$ and $\beta\in (G'/G'')\otimes \K$.
\end{lemma}

\begin{proof}  
It is enough to show that 
$\kk ''(a \cdot j(b)  \cdot a ^{-1})=e^{p(a) \otimes 1} \cdot \kk ''(j(b))$ 
for $a \in G/G''$ and $b \in G'/G''$. To check this equality, recall the 
well-known conjugation formula in exponential groups 
(see Lazard \cite{L}), which in our situation says
\[
xyx ^{-1}=\exp(\ad_x)(y) ,
\]
for $x,y \in \eexp(E/\overline{E''})$. 
Hence $\kk ''(a \cdot j(b)  \cdot a ^{-1})=e^{\ad_{\kk ''(a)}}(\kk ''(j(b)))$, 
which equals $e^{p(a) \otimes 1} \cdot \kk ''(j(b))$, since
$\kk ' \circ p(a)= p(a) \otimes 1$.
\end{proof}

Recall that  $B_G \otimes \K$ is a module over $\K G_{\ab}$, 
and $\overline{E'}/\overline{E''}$ is a module over 
$\widehat { \K[X]}=\K[[X]]$. 
Shift the Malcev filtration on $\overline{E'}/\overline{E''}$
by setting $F'_q\, \overline{E'}/\overline{E''} :=F_{q+2}\,  
\overline{E'}/\overline{E''}$, for each $q\ge 0$. 

\begin{prop} 
\label{lem=filt}
The $\K$-linear map
$\kk _0 \otimes \K\colon B_G \otimes \K \to \overline{E'}/\overline{E''}$ 
induces a filtered $\widehat{\exp}$-linear isomorphism between 
$\widehat { B_G \otimes \K}$, endowed with the 
filtration coming from the $I$-adic completion, and 
$\overline{E'}/\overline{E''}$, endowed with the shifted 
Malcev filtration $F'$.
\end{prop}

\begin{proof} 
We start by proving that
\begin{equation} 
\label{eq=kkinc}
\kk _0  \otimes \K\,(I^q B_G \otimes \K) \subset 
F'_q \, \overline{E'}/\overline{E''}
\end{equation}
for all $q \geq 0$. First note that $\overline{E'}/\overline{E''}= 
F_2 \overline{E'}/\overline{E''} $, by filtered exactness of 
\eqref{eq=hsq}, together with \eqref{eq=griso}.
Next, recall that $\exp(I) \subset (X)$. 
Finally, note that $(X)^rF_s\, \overline{E'}/\overline{E''} 
\subset F_{r+s}\,\overline{E'}/\overline{E''} $
for all $r,s$.   These observations, together with 
the $\exp$-equivariance property from Lemma \ref{lem=eq}, 
establish the claim.

In view of  \eqref{eq=kkinc}, $\kk _0  \otimes \K$ induces 
a filtered, $\widehat{\exp}$-linear  map  
from $\widehat { B_G \otimes \K}$ to $\overline{E'}/\overline{E''}$.  
We are left with checking  this map is a filtered isomorphism. 
For that, it is enough to show
\begin{equation} 
\label{eq=g1}
\gr^q(\kk _0  \otimes \K)\colon \gr^q_I( B_G \otimes \K) 
\longrightarrow \gr_F^{q+2}(\overline{E'}/\overline{E''})
\end{equation}
is an isomorphism, for each $q \geq 0$. 

Recall from \eqref{eq=k0diagram} that 
$\kk '' \circ j= \iota \circ \kk _0$. 
By a result of W.~Massey \cite[pp.~400--401]{Ma}, the 
map $j\colon G'/G''\to G/G''$ induces isomorphisms
\begin{equation} 
\label{eq=g2}
\gr^q(j)  \otimes \K \colon \gr^q_I( B_G \otimes \K) 
\xrightarrow{\,\simeq\,}   \gr^{q+2}  (G / G '') \otimes \K  
\end{equation}
for all $q \geq 0$. Since $\kk''$ is a 
Malcev completion, it induces isomorphisms
\begin{equation} 
\label{eq=g3}
\gr^q(\kk '')\colon  \gr^{q  }  (G / G '') \otimes 
\K \xrightarrow{\,\simeq\,} \gr_F^{q}(E/\overline{E''})   
\end{equation}
for all $q \geq 1$.  Finally, the inclusion map $\iota$ induces isomorphisms 
$\gr_F^q(\iota)$,  for all $q\ge 2$, again by filtered exactness 
of \eqref{eq=hsq}, since $X$ is abelian. This finishes the proof.
\end{proof}

\subsection{Alexander invariants of $1$-formal groups}
\label{ss53}

The next result is a new, explicit manifestation of D.~Sullivan's 
general philosophy, according to which the algebraic topology 
of a formal space in characteristic zero is determined by the 
cohomology ring.

\begin{theorem} 
\label{thm=complalex}
Let $G$ be a $1$-formal group. Then the $I$-adic completion of 
the Alexander invariant, $\widehat { B_G  \otimes \K}$, 
is isomorphic to the $(X)$-adic completion of the 
infinitesimal Alexander invariant, $\widehat { B_{\h (G)}}$,
by a filtered $\widehat {\exp}$-linear isomorphism.
\end{theorem}

\begin{proof}
Set $\h:=\h(G)$. Consider the following sequence of Lie algebras, 
\[
\xymatrix{0\ar[r]& \h'/\h'' \ar[r]& \h/\h'' \ar[r]& \h/\h' \ar[r]& 0}. 
\]
Clearly, $\h/\h'$ and $\h'/\h''$ are abelian, and $\h/\h' \cong X$. 
Moreover, the sequence is exact in each degree.

Assign degree $q$ to $\bigwedge ^qX$ and degree $2$ to 
$Y$ in \eqref{eq=infalex}. Then $ \h' / \h ''$, viewed as a 
graded  $\K[X]$-module via the above exact sequence, 
is graded isomorphic to the  $\K[X]$-module $\coker (\nabla)$, 
see  \cite[Theorem 6.2]{PS2}. Taking $(X)$-adic completions, 
the claim follows from Proposition \ref{lem=filt}, since 
$\overline{E'}/\overline{E''}\cong \widehat{\h' / \h ''}$, by 
$1$-formality.
\end{proof}

\begin{remark}
\label{rem=altpf}
An alternative proof of Theorem \ref{thm=tcfintro}, in the case 
$\rho=\id_{\GL_1(\C)}$, based on Theorem \ref{thm=complalex}, 
goes as follows. Consider the Alexander invariant,
$B_G\otimes \C$, over $\C G_{\ab}$, and the infinitesimal Alexander 
invariant, $B_{\h(G)}$, over $\C[G_{\ab}\otimes \C]$. The varieties 
defined by the Fitting ideals of these two modules coincide, away 
from the origin, with $\V_k(G)$, and $\R_k(G)$ respectively. By 
base change and  Theorem \ref{thm=complalex}, the Fitting ideals are
identified in the $1$-formal case via $\widehat{\exp}$, upon completion. 
Since $\C[[X]]$ is faithfully flat over $\C \{ X\}$, the analytic germs of 
the corresponding Fitting loci are identified, via the local exponential 
isomorphism.
\end{remark}

\begin{remark}
\label{rem=arr}
Suppose $M$ is the complement of a hyperplane arrangement 
in $\C^{m}$, with fundamental group $G=\pi_1(M)$, and $\rho=\id_{\GL_1(\C)}$.  
In this case, Theorem \ref{thm=tcfintro} can be deduced from 
results of Esnault--Schechtman--Viehweg \cite{ESV} and 
Schechtman--Terao--Varchenko \cite{STV}.  In fact, one 
can show that there is a combinatorially defined open 
neighborhood $U$ of $0$ in $H^1(M, \C)$ with the 
property that  
$H^*(M,{}_{\rho}\C) \cong H^*(H^{\bullet}(M,\C),  \mu_z)$, 
for all $z\in U$,  where $\rho=\exp (z)$. 
A similar approach works as soon as $W_1 (H^1 (M,\C))=0$,  
in particular, for complements of arrangements 
of hypersurfaces in projective or affine space.   
For details, see \cite[Corollary 4.6]{DM}.
\end{remark}

The local statement from Theorem \ref{thm=tcfintro} is the 
best one can hope for, as shown by the following classical example.

\begin{example} 
\label{ex:knots}
Let $M$ be the complement in $S^3$ of a tame knot $K$. 
Since $M$ is a homology circle, it follows easily that $M$ is 
a formal space; therefore, its fundamental group,  
$G=\pi_1(M)$, is a $1$-formal group. Let $\Delta(t)\in \Z [t, t^{-1}]$ 
be the Alexander polynomial of $K$. It is readily seen that 
$\V_1(G)= \{ 1\} \coprod \Zero (\Delta)$ and $\R_1(G)= \{ 0\}$. 
Thus, if $\Delta(t)\not\equiv 1$, then $\exp(\R_1(G)) \ne \V_1(G)$. 

Even though the germ of $\V_1(G)$ at $1$ provides no information 
in this case, the global structure of $\V_1(G)$ is quite meaningful. 
For example, if $K$ is an algebraic knot, then  $\Delta(t)$ must be  
product of cyclotomic polynomials, as follows from work of 
Brauner and Zariski from the 1920s, see \cite{Mi}.
\end{example}

\begin{remark} 
\label{rem:curves}
Let $M$ be the complement in $\C^2$ of an 
algebraic curve, with fundamental group $G=\pi_1(M)$, 
and let $\Delta(t)\in \Z [t, t^{-1}]$ be the Alexander polynomial 
of the total linking cover, as defined by Libgober; see \cite{Li94} 
for details and references. It was shown in \cite{Li94} that 
all the roots of $\Delta(t)$ are roots of unity. This 
gives restrictions on which finitely presented groups can be 
realized as fundamental groups of plane curve complements. 

Let $\Delta^G \in \Z [t_1^{\pm 1},\dots, t_n^{\pm 1}]$ be the
multivariable Alexander polynomial of an arbitrary quasi-projective 
group.  Starting from Theorem \ref{thm=posobstr}\eqref{a2}, we 
prove in \cite{DPS-codone} that $\Delta^G$ must have a single 
essential variable, if $n\ne 2$. Examples from \cite{P07} show 
that this new obstruction efficiently detects non-quasi-projectivity 
of (local) algebraic link groups.  Note that all roots of the 
one-variable (local) Alexander polynomial $\Delta (t)$ of 
algebraic links are roots of unity; see \cite{Mi}.
\end{remark}

\section{Position and resonance obstructions} 
\label{sec=posobs}

In this section, we prove Theorem \ref{thm=posobstr} from 
the Introduction. The basic tool is a result of Arapura \cite{A}, 
which we start by recalling.

\subsection{Arapura's theorem}
\label{subsec=arapura}
First, we establish some terminology.   By a {\em curve}\/  
we mean a smooth, connected, complex algebraic variety of 
dimension $1$.  A  curve $C$ admits a canonical compactification 
$\oC$, obtained by adding a finite number of points. 

Following  \cite[p.~590]{A}, we say a map  $f\colon M \to C$ from 
a connected, quasi-compact K\"{a}hler manifold $M$ to a 
curve $C$ is {\em admissible}\/ if $f$ is holomorphic
and surjective, and has a holomorphic, surjective extension with 
connected fibers, $\overline{f}\colon \oM \to \oC$, where $\oM$ 
is a smooth compactification, obtained by adding divisors with 
normal crossings.

With these preliminaries, we can state Arapura's result 
\cite[Proposition V.1.7]{A}, in a slightly modified form, 
suitable for our purposes.

\begin{theorem} 
\label{thm:vadm}
Let $M$ be a connected, quasi-compact K\"{a}hler manifold. 
Denote by $\{ \V^{\alpha}\}_{\alpha}$ the set of irreducible components 
of $\V_1(\pi_1(M))$ containing $1$. If  $\dim \V^{\alpha}>0$, then 
the following hold.

\begin{enumerate}
\item \label{va1} 
There is an admissible map, 
$f_{\alpha}\colon M \to C_{\alpha}$, where $C_{\alpha}$ 
is a smooth curve with $\chi (C_{\alpha})<0$, 
such that  
\[
\V^{\alpha}=f_{\alpha}^* \T_{\pi_1(  C_{\alpha} ) }
\]
and $(f_{\alpha}) _{\sharp}\colon \pi_1(M) \to \pi_1(  C_{\alpha}  )$ 
is surjective.

\item \label{va2}
There is an isomorphism
\[
H^1(M, {}_{f_{\alpha}^* {\rho}}\C)\cong H^1( C_{\alpha}, {}_{\rho}\C) ,
\]
for all except finitely many local systems 
${\rho} \in \T _{\pi_1(  C_{\alpha} )}$.
\end{enumerate}
\end{theorem}

When $M$ is compact, similar results to Arapura's were obtained 
previously by Beauville \cite{Beau2} and Simpson \cite{Sim}.
The closely related construction of regular mappings from 
an algebraic variety $M$ to a curve $C$ starting with suitable 
differential forms on $M$ goes back to Castelnuovo--de Franchis, 
see Catanese \cite{Cat}. When both $M$ and $C$ are compact, 
the existence of a non-constant holomorphic map $M \to C$ 
is closely related to the existence of an epimorphism 
$\pi_1(M) \surj \pi_1(C)$, see Beauville \cite{Beau1}
and Green--Lazarsfeld \cite{GL}. In the non-compact case, 
this phenomenon is discussed in Corollary V.1.9 from \cite{A}.

\subsection{Isotropic subspaces}
\label{subsec=isotropic}
Before proceeding, we introduce some notions which will 
be of considerable use in the sequel. 
Let $\mu \colon H^1 \wedge H^1 \to H^2$ be a $\C$-linear map, 
and $\R_k(\mu)\subset H^1$ be the corresponding resonance varieties, 
as defined in Remark \ref{rem=murelres}. One way to construct elements 
in these varieties is as follows.

\begin{lemma} 
\label{lem=linalg1}
Suppose $V \subset H^1$ is a linear subspace of dimension $k$. 
Set $i=\dim \im (\mu\colon V \wedge V \to H^2)$.
If $i<k-1$, then $V \subset \R_{k-i-1}(\mu)\subset \R_1(\mu)$.
\end{lemma}

\begin{proof}
Let $x \in V$, and set $x^{\perp}_V=\{y \in V \mid \mu (x \wedge y)=0\}$.  
Clearly, $\dim x^{\perp}_V \geq k - i$. On the other hand, 
$x^{\perp}_V \slash \C \cdot x \subset H^1(H^{\bullet},\mu _x)$, 
and so $x \in  \R_{k-i-1} (\mu)$. 
\end{proof}

Therefore, the subspaces $V \subset H^1$ for which 
$\dim \im (\mu\colon V \wedge V \to H^2)$ is small
are particularly interesting. This remark gives a 
preliminary motivation for the following key definition.

\begin{definition} 
\label{def=position} 
Let $\mu \colon \bigwedge ^2H^1 \to  H^2$ be a 
$\C$-linear mapping, where $\dim H^1<\infty$, and let 
$V \subset H^1$ be a $\C$-linear subspace. 

\begin{romenum}
\item \label{it1}
$V$ is {\em  $0$-isotropic} (or, {\em isotropic}) 
with respect to $\mu$ if the restriction 
$\mu ^V\colon \bigwedge ^2V \to H^2$ is trivial. 

\item \label{it2}
$V$ is  {\em $1$-isotropic}\/ with respect to $\mu$ if the restriction 
$\mu ^V\colon\bigwedge ^2V \to  H^2$ has $1$-dimensional 
image and is a non-degenerate skew-symmetric bilinear form.
\end{romenum}
\end{definition}

\begin{example} 
\label{ex=position} 
Let $C$ be a smooth curve, 
and let $\mu =\cup_C \colon \bigwedge ^2H^1(C,\C) \to H^2(C,\C)$ be 
the usual cup-product map.  There are two cases of interest to us.
\begin{romenum}
\item \label{curv1}
 If  $C$ is not compact, then $ H^2(C,\C)=0$ 
and so any subspace $V \subset H^1(C,\C)$ is isotropic.
\item \label{curv2} 
If $C$ is  compact, of genus $g\ge 1$, then 
$ H^2(C,\C)=\C$ and $ H^1(C,\C)$ is  $1$-isotropic.
\end{romenum}
\end{example}

Now let $\mu _1\colon  \bigwedge ^2H^1_1 \to  H^2_1$ and 
$\mu _2\colon  \bigwedge ^2H^1_2 \to  H^2_2$ be two $\C$-linear maps. 
\begin{definition} 
\label{def=similar} 
The maps $\mu _1$ and  $\mu _2$ are {\em equivalent} 
(notation $\mu _1 \simeq   \mu _2$) if there exist linear 
isomorphisms $\phi ^1\colon H_1^1 \to H_2^1$ and 
$\phi ^2\colon \im (\mu _1) \to \im (\mu _2)$ 
such that $\phi^2 \circ \mu _1= \mu_2 \circ \wedge ^2 \phi^1$.
\end{definition} 

The key point of this definition is that the $k$-resonant varieties 
$\R_k(\mu _1)$ and $\R_k(\mu _2)$ can be identified under $\phi ^1$ 
when $\mu _1 \simeq \mu _2$.  Moreover, subspaces that are either 
$0$-isotropic or $1$-isotropic with respect to $\mu_1$ and $\mu_2$ 
are matched under $\phi^1$.

\subsection{Proof of Theorem \ref{thm=posobstr}\eqref{a1}}
\label{ss63}

We have 
$\V^{\alpha}= f_{\alpha}^*\T_{\pi_1(C_{\alpha})}$, where $f_{\alpha}$ is
admissible and $\chi(C_{\alpha})<0$; see Theorem \ref{thm:vadm}\eqref{va1}.
Therefore, $\TT^{\alpha}= f_{\alpha}^* H^1(C_{\alpha}, \C)$. If the curve
$C_{\alpha}$ is non-compact, the subspace $\TT^{\alpha}$ is clearly isotropic,
and $\dim \TT^{\alpha}=b_1(C_\alpha) \ge 2$. If $C_{\alpha}$ is compact and
$f_{\alpha}^*$ is zero on $H^2(C_{\alpha}, \C)$, we obtain 
the same conclusion as before. Finally, if $C_{\alpha}$ is compact and
$f_{\alpha}^*$ is non-zero on $H^2(C_{\alpha}, \C)$, then plainly
$\TT^{\alpha}$ is $1$-isotropic and $\dim \TT^{\alpha}=b_1(C_\alpha) \ge 4$.
The isotropicity property is thus established.

\subsection{Genericity obstruction}
\label{ss64}

The next three lemmas will be used in establishing the position 
obstruction from Theorem \ref{thm=posobstr}\eqref{a2}.

\begin{lemma}
\label{lem=fcok} 
Let $X$ be a connected quasi-compact K\"{a}hler manifold, $C$ a 
smooth curve and $f \colon X\to C$ a non-constant holomorphic 
mapping.  Assume that $f$ admits a holomorphic extension 
$\hat f\colon \hat X \to \hat C$, where $\hat X$ (respectively, $\hat C$) 
is a smooth compactification of $X$ (respectively, $C$). Then the induced 
homomorphism in homology, $f_*\colon H_1(X, \Z)\to H_1(C, \Z)$, 
has finite cokernel.
\end{lemma}

\begin{proof}
Let $Y =\Sing(\hat f)$ be the set of singular points of $\hat f$, i.e., 
the set of all points $x \in \hat X$ such that $d_x\hat f=0$. Then 
$Y$ is a closed analytic subset of $\hat X$. Using Remmert's Theorem, 
we find that $Z=\hat f (Y)$ is a closed analytic subset of $\hat C$,  
and $\hat f(\hat X)=\hat C$.  By Sard's Theorem, $Z \ne \hat C$, 
hence $Z$ is a finite set. 

Let $B=(\hat C \setminus C) \cup Z$; set $C'=\hat C \setminus B$, 
and $\hat X '=\hat X \setminus \hat f^{-1}(B)   =\hat f^{-1}(C')$. Then 
the restriction $\hat f '\colon \hat X '\to C'$ is a locally trivial fibration;  
its fiber is a compact manifold, and thus has only finitely many 
connected components. Using the tail end of the homotopy 
exact sequence of this fibration, we deduce 
that the induced homomorphism, 
$\hat f'_{\sharp} \colon \pi_1(\hat X ')\to \pi_1(C')$, 
has image of finite index. 

Now note that $i \circ \hat f ' \circ k=f\circ j$, where 
$i\colon C' \to C$, $j\colon X \setminus \hat f^{-1}(B) \to X$, 
and $k\colon X \setminus \hat f^{-1}(B) \to \hat X '$ are the 
inclusion maps. From the above, it follows that
$f_{*}\colon H_1(X,\Z)\to H_1(C,\Z)$ has image 
of finite index. 
\end{proof}

Let $\V^{\alpha} \ne \V^{\beta}$ be two  positive-dimensional, 
irreducible components of $\V_1(\pi_1(M))$ containing $1$. Realize 
them by pull-back, via admissible maps, $f_{\alpha}\colon M\to C_{\alpha}$ and 
$f_{\beta}\colon M\to C_{\beta}$, as in Theorem \ref{thm:vadm}\eqref{va1}.
We know that generically (that is, for $t\in C_{\alpha}\setminus B_{\alpha}$, 
where $B_{\alpha}$ is finite) the fiber $f_{\alpha}^{-1}(t)$  is smooth and 
irreducible.

\begin{lemma}
\label{lem=noncst}
In the above setting, there exists $t\in C_{\alpha}\setminus B_{\alpha}$
such that the restriction of $f_{\beta}$ to $f_{\alpha}^{-1}(t)$ 
is non-constant.
\end{lemma}

\begin{proof}
Assume $f_{\beta}$ has constant value, $h(t)$, on the fiber 
$f_{\alpha}^{-1}(t)$, for $t\in C_{\alpha}\setminus B_{\alpha}$. 
We first claim that this implies the existence of a continuous 
extension, $h\colon C_{\alpha}\to C_{\beta}$, with the property 
that $h\circ f_{\alpha}= f_{\beta}$. 

Indeed, let us pick an arbitrary special value, $t_0\in B_{\alpha}$, 
together with a sequence of generic values, 
$t_n\in C_{\alpha}\setminus B_{\alpha}$, 
converging to $t_0$. For any $x\in f_{\alpha}^{-1}(t_0)$, 
note that the order at $x$ of the holomorphic function $f_{\alpha}$ 
is finite. Hence, we may find a sequence, $x_n \to x$, such that 
$f_{\alpha}(x_n)=t_n$. By our assumption, $f_{\beta}(x)= \lim h(t_n)$, 
independently of $x$, which proves the claim.

At the level of character tori, the fact that 
$h\circ f_{\alpha}= f_{\beta}$ implies  
$\V^{\beta}= f_{\beta}^* \T_{\pi_1(C_{\beta})}\subset 
f_{\alpha}^*\T_{\pi_1(C_{\alpha})}= \V^{\alpha}$, a contradiction.
\end{proof}

\begin{lemma}
\label{lem=finite}
Let $\V^{\alpha}$ and $\V^{\beta}$ be two distinct 
irreducible components of $\V_1(\pi_1(M))$ containing $1$.  
Then $\V^{\alpha}\cap \V^{\beta}$ is finite.
\end{lemma}

\begin{proof}
We may suppose that both components are positive-dimensional. 
Lemma \ref{lem=noncst} guarantees the existence of a 
generic fiber of $f_{\alpha}$, say $F_{\alpha}$, with the property 
that the restriction of $f_{\beta}$ to $F_{\alpha}$, call it 
$g\colon F_{\alpha} \to C_{\beta}$, is non-constant. By 
Lemma \ref{lem=fcok}, there exists a positive integer $m$
with the property that
\begin{equation}
\label {eq=cokf}
m\cdot H_1(C_{\beta}, \Z)\subset \im (g_*)\, .
\end{equation}
We will finish the proof by showing that
$\rho^m =1$, for any $\rho\in \V^{\alpha}\cap \V^{\beta}$.

To this end, write $\rho= \rho_{\beta}\circ (f_{\beta})_*$, 
with $\rho_{\beta}\in \T_{\pi_1(C_{\beta})}$. For an 
arbitrary element $a\in H_1(M, \Z)$, we have 
$\rho^m (a)= \rho_{\beta}(m\cdot (f_{\beta})_*a)$. 
From \eqref{eq=cokf}, it follows that 
$m\cdot (f_{\beta})_*a=(f_{\beta})_*( j_\alpha)_*b$, 
for some $b\in H_1(F_{\alpha}, \Z)$, where 
$j_\alpha\colon F_{\alpha}\hookrightarrow M$ is 
the inclusion. 
On the other hand, we may also write 
$\rho= \rho_{\alpha}\circ (f_{\alpha})_*$, with 
$\rho_{\alpha}\in \T_{\pi_1(C_{\alpha})}$. Hence,  
$\rho^m (a)= \rho_{\alpha}((f_{\alpha})_* (j_{\alpha})_*b)= 
\rho_{\alpha}(0)=1$, as claimed.
\end{proof}

By passing to tangent spaces, Lemma \ref{lem=finite} implies 
Theorem \ref{thm=posobstr}\eqref{a2}.

\subsection{Proof of Theorem \ref{thm=posobstr}\eqref{a0}}
\label{ss65}

For this property, only $1$-formality is needed. See 
Theorem \ref{thm=bnew}\eqref{bn3}.

\subsection{Proof of Theorem \ref{thm=posobstr}\eqref{a3}}
\label{ss66}

This filtration-by-dimension resonance obstruction is a 
consequence of position obstructions from Parts \eqref{a1} 
and \eqref{a2}, in the presence of $1$-formality. We may assume 
$k<b_1(G)$, since otherwise $\R_k(G)= \{0\} $, and there is nothing 
to prove.

To prove the desired equality, we have to check first that any 
non-zero element $u \in \R_k(G)$ belongs to some 
$ \R^{\alpha}$ with $\dim \R^{\alpha} >k$. 
The definition of $\R_k$ guarantees the existence of 
elements $v_1,\dots ,v_k \in H^1(G,\C)$ with $v_i \cup  u=0$, 
and such that $u, v_1,\dots ,v_k$ are linearly independent. 
Since the subspaces  $\langle u, v_i \rangle $ spanned by the 
pairs $\{u,v_i\}$ are clearly contained in $\R_1(G)$, it follows that 
$\langle u,v_i \rangle\subset  \R^{\alpha _i}$.  
Necessarily $ \alpha _1=\cdots =\alpha _k :=\alpha$, since 
otherwise property \eqref{a2} would be violated. This proves 
$u\in \R^{\alpha}$, with $\dim  \R^{\alpha} >k$. 
Now, if $p({\alpha})=1$, then $\dim \R^{\alpha}>k+1$.  
For otherwise, we would have $\R^{\alpha}= u^{\perp}_{ \R^{\alpha}}$, 
which would violate the non-degeneracy property from 
Definition \ref{def=position}\eqref{it2}.
Finally, that $\dim  \R^{\alpha} >k+ p({\alpha})$ implies 
$\R^{\alpha} \subset  \R_k(G)$ follows at once  
from Lemma \ref{lem=linalg1}.

\subsection{Fibered quasi-K\"{a}hler groups}
\label{ss67}

The next Lemma proves Theorem \ref{thm=posobstr}\eqref{a4}.

\begin{lemma}
\label{cor:fibered}
Let $M$ be a connected  quasi-compact K\"{a}hler manifold. 
Suppose the group $G=\pi_1(M)$ is $1$-formal. 
The following are then equivalent.
\begin{romenum}
\item \label{gf3}
There is an epimorphism, $\varphi\colon G\surj \bF_r$, onto
a free group of rank $r\ge 2$.
\item \label{gf4}
The resonance variety $\R_1(G)$ strictly contains $\{ 0\}$.
\end{romenum}
\end{lemma}

\begin{proof}
The implication \eqref{gf3} $\Rightarrow$
\eqref{gf4}  holds in general; this may be seen by 
using the $\cup_G$-isotropic subspace 
$\varphi^* H^1(\bF_r, \C)\subset H^1(G, \C)$.
If $G$ is $1$-formal and $\R_1(G)$ strictly contains $\{ 0\}$, 
there is a positive-dimensional component $\V^{\alpha}$, by 
Theorem \ref{thm=posobstr}\eqref{a0}.  Applying 
Theorem \ref{thm:vadm}, we obtain an admissible map, 
$f\colon M\to C$, onto a smooth complex curve $C$ with 
$\chi(C)<0$. Property \eqref{gf3} follows then from 
\cite[Corollary V.1.9]{A}.
\end{proof}

This finishes the proof of Theorem \ref{thm=posobstr} 
from the Introduction. 

We close this section with a pair of examples showing that 
both the quasi-K\"{a}hler  and the $1$-formality 
assumptions are needed in order for Lemma \ref{cor:fibered}
to hold.  

\begin{example}
\label{ex:heisenberg}
Consider the smooth, quasi-projective variety $M_1$
examined by Morgan in \cite[p.203]{M} (the complex version of the 
Heisenberg manifold). It is well-known that $G_1=\pi_1(M_1)$ is a 
nilpotent group.   Therefore, property \eqref{gf3} fails for $G_1$. 
Nevertheless, $\R_1(G_1)=\C^2$, and so property \eqref{gf4} holds
for $G_1$. In particular, the group $G_1$ is not $1$-formal.
\end{example}

\begin{example}
\label{ex:nonorientable}
Let $N_h$ be the non-orientable  surface of genus 
$h\ge 1$, that is, the connected sum of $h$ real projective planes. 
It is readily seen that $N_h$ has the rational homotopy type of 
a wedge of $h-1$ circles.  Hence $N_h$ is a formal space, 
and so $\pi_1(N_h)$ is a $1$-formal group.   
Moreover, $\R_1(\pi_1(N_h))=\R_1(\bF_{h-1})$, and so 
$\R_1(\pi_1(N_h))=\C^{h-1}$, provided $h\ge 3$.  
Thus, property \eqref{gf4} holds for all groups 
$\pi_1(N_h)$ with $h\ge 3$. 

Now suppose there is an epimorphism 
$\varphi\colon\pi_1(N_h) \surj \bF_r$ with $r\ge 2$, as in \eqref{gf3}.  
Then the subspace 
$\varphi^* H^1(\bF_r, \Z_2)\subset H^1(\pi_1(N_h), \Z_2)$ 
has dimension at least $2$, and is isotropic with respect to 
$\cup_{\pi_1(N_h)}$.  Hence, $h\ge 4$, by Poincar\'e duality 
with $\Z_2$ coefficients.

Focussing on the case $h=3$, we see that the group $\pi_1(N_3)$ 
is $1$-formal, yet the implication \eqref{gf4} $\Rightarrow$ \eqref{gf3} 
from Lemma \ref{cor:fibered} fails for  this group.
It follows that $\pi_1(N_3)$ cannot be realized as the fundamental 
group of a  quasi-compact K\"{a}hler manifold. Note 
that this assertion is {\em not}\/ a consequence of Theorem
\ref{thm=posobstr}, Parts \eqref{a1}, \eqref{a2} and \eqref{a3}. 
Indeed, $\cup_{\pi_1(N_h)} \simeq \cup_{\bF_{h-1}}$ (over $\C$), 
while $\bF_{h-1}=\pi_1(\PP^1 \setminus \{\text{$h$ points}\})$, 
for all $h\ge 1$.
\end{example}

\section{Regular maps onto curves and isotropic subspaces} 
\label{sec=curves}

In this section, we present a couple of useful complements
to Theorem \ref{thm=posobstr}. 

\subsection{Admissible maps and isotropic subspaces}
\label{subsec=admiso}

We now consider in more detail which admissible maps 
$f_{\alpha}\colon M \to C_{\alpha}$ may occur in 
Theorem \ref{thm:vadm}.

\begin{prop} 
\label{prop=posobstr} 
Let $M$ be a connected quasi-compact K\"{a}hler manifold, 
and let $f\colon M \to C$ be an admissible map onto the 
smooth curve $C$.
\begin{enumerate}

\item \label{ra1} 
If $W_1(H^1(M,\C))=H^1(M,\C)$, then the curve $C$ 
is either compact, or it is obtained from a compact smooth 
curve $\oC$ by deleting a single point.

\item \label{ra2} 
If  $W_1(H^1(M,\C))=0$, then  the curve $C$ is rational.
If $\chi(C)<0$, then $C$ is obtained from $\C$ by deleting 
at least two points, and $f^* H^1( C,\C)$ is $0$-isotropic with 
respect to $\cup_M$.

\item \label{ra3} 
Assume in addition that $\pi_1(M)$ is $1$-formal.
If the curve $C$ is compact of genus at least $1$, then  
$f^*\colon H^2( C,\C) \to H^2( M ,\C)$ 
is injective, and so $f^* H^1( C,\C)$ is 
$1$-isotropic with respect to $\cup_M$.

\end{enumerate}
\end{prop}

\begin{proof}
Recall that $f _{\sharp}\colon \pi_1(M) \to \pi_1(  C )$ 
is surjective; hence $f^*\colon H^1( C,\C) \to 
H^1( M,\C)$ is injective. 

Part~\eqref{ra1}. A quasi-compact K\"{a}hler manifold $M$ 
inherits a mixed Hodge structure from each good compactification 
$\oM$, by Deligne's construction in the smooth quasi-projective 
case \cite{D}. Furthermore, if $M$ is quasi-projective, this structure 
is unique, as shown in \cite[Theorem 3.2.5(ii)]{D}.

By the admissibility condition on $f\colon M\to C$, there is 
a good compactification ${\overline M}$ such that $f$ extends 
to a regular morphism ${\overline f}\colon \oM \to \oC$. 
Fixing such an extension, the condition 
$W_1(H^1(M,\C))=H^1(M,\C)$ just means that
$j^*(H^1(\oM,\C))=H^1(M,\C)$, where $j\colon 
M \to \oM$ is the inclusion. Since regular maps 
$f$ which extend to good compactifications of source 
and target obviously preserve weight filtrations, the 
mixed Hodge structure on $H^1(C,\C)$ must be pure 
of weight $1$, see \cite{D}. If we write 
$C= \oC \setminus A$, for some finite set $A$, 
then there is an exact Gysin sequence
\[
\xymatrixcolsep{12pt}
\xymatrix{0 \ar[r]& H^1(\oC, \C) \ar[r]& 
H^1({ C}, \C) \ar[r]& H^0(A,\C)(-1) 
\ar[r]&  H^2(\oC, \C) \ar[r]& 
H^2({ C}, \C) \ar[r]& 0 },
\]
see for instance \cite[p.~246]{D1}. But $H^0(A,\C)(-1)$ is pure 
of weight $2$, and so $ H^1({ C}, \C)$ is pure of weight $1$ 
if and only if $\abs{A} \leq 1$.

\smallskip
Part~\eqref{ra2}. 
By the same argument as before, we infer in this 
case that $H^1( C,\C)$ should be pure of weight $2$.
The above Gysin sequence shows that $ H^1({ C}, \C)$ 
is pure of weight $2$ if and only if $g(\oC)=0$, 
i.e., $\oC=\PP^1$. Finally,  $\chi (C)<0$ implies 
$\abs{A} \ge 3$. 

\smallskip
Part~\eqref{ra3}. Set $G:=\pi_1(M)$, $\T_M:=\T_G$ and 
$\T:=\T_{\pi_1(C)}$. Note  that $\dim \T>0$. 
Furthermore, the character torus $\T$ is embedded in 
$\T_M$, and its Lie algebra $T_1 (\T)$ is embedded in 
$T_1(\T_M)$, via the natural maps induced by $f$.  
By Theorem \ref{thm:vadm}\eqref{va2}, 
\[
\dim H^1(M,{}_{f^*{\rho}}\C)= \dim H^1( C, {}_{\rho}\C) ,
\]
for $\rho \in \T$ near $1$ and different from $1$, since
both the surjectivity of $f_{\sharp}$ in Part \eqref{va1}, 
and the property from Part \eqref{va2} do not require 
the assumption $\chi (C)<0$.

Applying Theorem \ref{thm=tcfintro} to both $G$ (using our 
$1$-formality hypothesis), and $\pi_1(C)$ (using 
Example \ref{ex:w2formal}), we obtain from the above equality that
\[
\dim H^1(H^{\bullet}(M,\C), \mu_{f^* z})= 
\dim H^1(H^{\bullet} (C,\C), \mu_z) ,
\]
for all $z\in H^1(C, \C)$ near $0$ and different from $0$.
Moreover, for any such $z$, a standard calculation shows   
$\dim H^1(H^{\bullet} (C,\C), \mu_z)=2g-2$, 
where $g=g(C)$. 

Now suppose $f^*\colon H^2( C,\C) \to H^2( M ,\C)$ were 
not injective. Then  $f^* H^1 ( C,\C)$ would be a $0$-isotropic 
subspace of $H^1(M, \C)$, containing $f^* (z)$.  In turn, this 
would imply  $\dim H^1(H^{\bullet}(M,\C), \mu_{f^* z})\ge 2g-1$, 
a contradiction.
\end{proof}

Using Theorem \ref{thm=posobstr}, we obtain the 
following. 

\begin{corollary} 
\label{cor:firstex}
Let $M$ be a connected  quasi-compact K\"{a}hler manifold, 
with fundamental group $G$, and first resonance variety 
$\R_1(G)=\bigcup_{\alpha}  \R^{\alpha}$. Assume $b_1(G)>0$ and
$\R_1(G)\ne \{0\}$.
\begin{enumerate}
\item \label{appl1} 
If $M$ is compact then $G$ is $1$-formal, and each component 
$\R^{\alpha}$ is $1$-isotropic, with $\dim \R^{\alpha}= 2g_{\alpha}\ge 4$.

\item \label{appl2} 
If $W_1(H^1 (M,\C))=0$ then $G$ is $1$-formal, and  each component 
$\R^{\alpha}$  is  $0$-isotropic, with $\dim \R^{\alpha}\ge 2$.

\item \label{appl3} 
If $W_1(H^1 (M,\C))=H^1 (M,\C)$ and $G$ is $1$-formal, then
$\dim \R^{\alpha}= 2g_{\alpha}\ge 2$, for all $\alpha$.
\end{enumerate}
\end{corollary}

\subsection{Cohomology in degree two}
\label{ss72}

We  point out the subtlety of the 
injectivity property from Proposition \ref{prop=posobstr}\eqref{ra3}.

\begin{example} 
\label{ex:notinj}
Let $L_g$ be the complex algebraic line bundle associated 
to the divisor given by a point on a projective smooth complex 
curve $C_g$ of genus $g\ge 1$. Denote by $M_g$ the total 
space of the $\C^*$-bundle associated to $L_g$. Clearly, 
$M_g$ is a smooth, quasi-projective manifold. 
(For $g=1$, this example was examined by
Morgan in \cite[p.~203]{M}.) Denote by $f_g \colon M_g\to C_g$ 
the natural projection.  This map is a locally trivial fibration, 
which is admissible in the sense of Arapura \cite{A}.
Since the Chern class $c_1(L_g)\in H^2(C_g, \Z)$ equals 
the fundamental class, it follows that
$f_g^* \colon H^2(C_g, \C)\to H^2(M_g, \C)$ is the zero map. 
Set $G_g= \pi_1(M_g)$.

A straightforward analysis of the Serre spectral sequence associated 
to $f_g$, with arbitrary untwisted field coefficients, shows that 
$(f_g)_* \colon H_1(M_g, \Z)\to H_1(C_g, \Z)$ is an isomorphism, 
which identifies the respective character tori, to be denoted in the 
sequel by $\T_g$. This also implies that $W_1(H^1 (M_g,\C))= 
H^1( M_g,\C)$, since this property holds for the compact variety $C_g$. 

We claim $f_g$ induces an isomorphism
\begin{equation} 
\label{eq:mgcg}
H^1(C_g, {}_{\rho}\C) \isom H^1(M_g, {}_{f_g^* {\rho}}\C) ,
\end{equation}
for all $\rho \in \T_g$. If $\rho=1$, this is clear. If $\rho \ne 1$, 
then $\Hom_{\Z \pi_1(C_g)}(\Z, {}_{\rho}\C)=0$, since the 
monodromy action of $\pi_1(C_g)$ on $\Z= H_1 (\C^*, \Z)$ 
is trivial. The claim follows from the $5$-term exact sequence 
for twisted cohomology associated to the group extension
$1\to \Z \to G_g \to \pi_1(C_g)\to 1$; 
see \cite[VI.8(8.2)]{HS}. 

It follows that
\begin{equation} 
\label{eq:vmg}
\V_k(G_g)=\begin{cases}
\T_g , & {\rm for}\ 0\le k\le 2g-2 ;\\
\{ 1\}  , & {\rm for}\   2g-1\le k\le 2g .
\end{cases}
\end{equation}
On the other hand, $\cup_{G_g}=0$, since $f_g^*=0$ on $H^2$. 
Therefore
\begin{equation} 
\label{eq:rmg}
\R_k(G_g)=\begin{cases}
T_1(\T_g), & {\rm for}\  0\le k\le 2g-1 ; \\
\{ 0\} ,        & {\rm for}\ k= 2g .
\end{cases}
\end{equation}

By inspecting \eqref{eq:vmg} and \eqref{eq:rmg}, 
we see that the tangent cone formula fails for $k=2g-1$. 
Consequently, the (quasi-projective) group $G_g$ 
cannot be $1$-formal. 
We thus see that the $1$-formality hypothesis from 
Proposition \ref{prop=posobstr}\eqref{ra3} is essential for 
obtaining the injectivity property of $f^*$ on $H^2$.
\end{example}

\begin{remark} 
\label{rk=posobstr} 
It is easy to show that $f^*\colon H^2( C,\C) \to H^2({M},\C)$ 
is injective when $M$ is compact. On the other hand, consider the following 
genus zero example, kindly provided to us by Morihiko Saito. Take 
$ C=\PP^1$ and  $M=\PP^1 \times \PP^1 \setminus (C_1 \cup C_2)$, 
where $C_1= \{\infty \}   \times  \PP^1 $ and $C_2$ is the diagonal in 
$\PP^1 \times \PP^1 $. The projection of $M$ on the first factor has as 
image $\C=\PP^1 \setminus \{\infty \}$ and affine lines as fibers; 
thus, $M$ is contractible. If we take $f\colon M \to \PP^1$ to be the map 
induced by the second projection, we get an admissible map such that 
$f^*\colon H^2( C,\C) \to H^2(M ,\C)$ is not injective. 
\end{remark}

\subsection{Twisted and Aomoto Betti numbers}
\label{ss73}

We now relate the dimensions of 
the cohomology groups $H^1(H^{\bullet} (M,\C), \mu_z)$ 
and $H^1(M, {}_{\rho}\C)$ corresponding to 
$z\in \Hom(G,\C)\setminus \{0\}$ and 
$\rho=\exp(z)\in \Hom(G,\C^*)\setminus \{1\}$, 
to the dimension and isotropicity of the resonance 
component $\R^{\alpha}$ to which $z$ belongs.

\begin{prop}
\label{prop=p4} 
Let $M$ be a connected  quasi-compact K\"{a}hler manifold, 
with fundamental group $G$, and first resonance variety 
$\R_1(G)=\bigcup_{\alpha}  \R^{\alpha}$.  If $G$ is $1$-formal, 
then the following hold. 

\begin{enumerate}
\item \label{f1}
If $z\in \R^{\alpha}$ and $z\ne 0$, then 
$\dim H^1(H^{\bullet} (M,\C), \mu_z)=\dim \R^{\alpha}-p(\alpha)-1$.

\item \label{f2}
If $\rho \in \exp (\R^{\alpha})$ and $\rho\ne 1$, then 
$\dim H^1(M, {}_{\rho}\C)\ge \dim \R^{\alpha}-p(\alpha)-1$, with 
equality for all except finitely many local systems $\rho$.
\end{enumerate}
\end{prop}

\begin{proof} 
Part~\eqref{f1}. 
Recall that $\R^{\alpha}=f_{\alpha}^* H^1(C_{\alpha}, \C)$. 
Exactly as in the proof of Proposition \ref{prop=posobstr}\eqref{ra3}, 
we infer that 
\begin{equation} 
\label{eq:dimgen}
\dim H^1(M, {}_{\rho}\C)= 
\dim H^1(H^{\bullet}(M,\C), \mu_z)= 
\dim H^1(H^{\bullet} (C_{\alpha},\C), \mu_{\zeta}) ,
\end{equation}
where $z=f_{\alpha}^* \zeta$ and $\rho =\exp (z)$,
for all $z\in \R^{\alpha}$ near $0$ and different from $0$. Clearly
\begin{equation} 
\label{eq:dimcurve}
\dim H^1(H^{\bullet} (C_{\alpha},\C), \mu_{\zeta})= 
\dim \R^{\alpha}- p({\alpha})-1 ,
\end{equation}
if $\zeta\ne 0$. Since plainly 
$\dim H^1(H^{\bullet}(M,\C), \mu_z)= 
\dim H^1(H^{\bullet}(M,\C), \mu_{\lambda z})$, 
for all $\lambda \in \C^*$,
equations \eqref{eq:dimgen} and \eqref{eq:dimcurve} 
finish the proof of \eqref{f1}.

Part~\eqref{f2}. 
Starting from the standard presentation of the group 
$\pi_1(C_{\alpha})$, a Fox calculus computation shows that 
$\dim H^1(C_{\alpha}, {}_{\rho'} \C)=\dim \R^{\alpha}- p({\alpha})-1$, 
provided $\rho' \ne 1$. By Theorem \ref{thm:vadm}\eqref{va2}, the 
equality $\dim H^1(M, {}_{\rho}\C)= \dim \R^{\alpha}- p({\alpha})-1$ 
holds for all but finitely many local systems $\rho\in \exp(\R^{\alpha})$. 
By semi-continuity, the inequality
$\dim H^1(M, {}_{\rho}\C)\ge \dim \R^{\alpha}- p({\alpha})-1$ 
holds for all $\rho \in \exp(\R^{\alpha})$. 
\end{proof}

\section{Arrangements of real planes}
\label{sec=realarr}

Let $\A=\{H_1,\dots,H_n\}$ be an arrangement of planes 
in $\RR^4$, meeting transversely at the origin.  By 
intersecting $\A$ with a $3$-sphere about $0$, we 
obtain a link $L$ of $n$ great circles in $S^3$. 
It is readily seen that the complement $M$ of the 
arrangement deform-retracts onto the complement 
of  the link.  Moreover, the fundamental group $G=\pi_1(M)$ 
has the structure of a semidirect product of free groups,  
$G=F_{n-1}\rtimes \Z$, and  $M$ is a $K(G,1)$. 
For details, see \cite{Z, MS0}.

\begin{example} 
\label{ex=2134} 
Let $\A=\A(2134)$ be the arrangement defined in complex 
coordinates on $\RR^4=\C^2$ by the half-holomorphic function 
$Q(z,w)=zw(z-w)(z-2\bar{w})$; see Ziegler \cite[Example 2.2]{Z}. 
Using a computation from \cite[Example 5.10]{MS0}, we obtain 
the following presentation for the fundamental group of the complement
\[
G=\langle x_1,x_2,x_3,x_4 \mid 
(x_1,x_3^2x_4),\: (x_2,x_4),\: (x_3,x_4)\rangle. 
\]
It can be seen that  
$E_G=\widehat{\bL}(x_1,x_2,x_3,x_4)/
\llangle 2 [x_1,x_3] + [x_1,x_4], [x_2,x_4], [x_3,x_4] \rrangle$,
where $\widehat{\bL}(X)$ is the Malcev Lie algebra obtained 
from $\bL^*(X)$ by completion with respect to the degree filtration, 
and $\llangle U \rrangle$ denotes the closed Lie ideal generated 
by a subset $U$.  Thus, $G$ is $1$-formal.
The resonance variety $\R_1(G)\subset \C^4$ has two 
components, $\R^{\alpha}=\{ x \mid x_4=0\}$ and 
$\R^{\beta}=\{ x \mid x_4+2 x_3=0\}$.  
The  resonance obstructions from Theorem \ref{thm=posobstr},
Parts \eqref{a1}, \eqref{a2} and \eqref{a3} are violated:
\begin{itemize}
\item The subspaces $\R^{\alpha}$ and  $\R^{\beta}$ 
are neither $0$-isotropic, nor $1$-isotropic.  
\item $\R^{\alpha}\cap \R^{\beta} = \{ x \mid x_3=x_4=0\}$, 
which is not equal to $\{0\}$. 
\item $\R_2(G)= \{x\mid x_1=x_3=x_4=0\} \cup 
\{x\mid x_2=x_3=x_4=0\}$, and neither of these components 
equals $\R^{\alpha}$ or $\R^{\beta}$.
\end{itemize}
Thus, $G$ is not the fundamental group of any smooth 
quasi-projective variety. 
\end{example}

Let $\A$ be an arrangement of transverse planes in $\RR^4$, 
with complement $M$. From the point of view of two classical 
invariants---the associated graded Lie algebra, and the 
Chen Lie algebra---the group $G=\pi_1(M)$ behaves like 
a $1$-formal group. Indeed,  the associated link $L$ 
has all linking numbers equal to $\pm 1$, in particular, 
the linking graph of $L$ is connected.  Thus, 
$\gr^*(G)\otimes \Q \cong \h_G$ and 
$\gr^*(G/G'')\otimes \Q \cong \h_G/\h''_G$, as graded 
Lie algebras, by \cite[Corollary 6.2]{MP} and 
\cite[Theorem 10.4(f)]{PS2}, respectively.  Nevertheless, 
our methods can detect non-formality, even in this 
delicate setting.

\begin{example} 
\label{ex=nonformal} 
Consider the arrangement $\A=\A(31425)$ defined 
in complex coordinates by the function 
$Q(z,w)=z(z-w)(z-2w)(2z+3w-5\overline{w}) (2z-w-5\overline{w})$; 
see \cite[Example 6.5]{MS1}. A computation shows that 
$TC_1(\V_2(G))$ has $9$ irreducible components, 
while $\R_2(G)$ has $10$ irreducible components; 
see \cite[Example 10.2]{MS2}, and \cite[Example 6.5]{MS1}, 
respectively.  By Theorem \ref{thm=tcfintro},  the group $G$ is 
not $1$-formal.  Thus, the complement $M$ cannot be a 
formal space, despite a claim to the contrary by 
Ziegler \cite[p.~10]{Z}.
\end{example}

\section{Wedges and products} 
\label{sec=wp}

In this section, we analyze products and coproducts of groups, 
together with their counterparts at the level of first 
resonance varieties.  Using our obstructions, 
we obtain conditions for realizability of free products 
of  groups by  quasi-compact K\"{a}hler manifolds. 

\subsection{Products, coproducts, and $1$-formality}
\label{subsec=cp1}

Let $\bF(X)$ be the free group on a finite set $X$, and let 
$ {\bL}^*(X)$ be the free Lie algebra on $X$, over a field $\K$ 
of characteristic $0$.  Denote by $\widehat {\bL}(X)$ the 
Malcev Lie algebra obtained from $ {\bL}^*(X)$ by completion 
with respect to the degree filtration. Define the group 
homomorphism $\kappa _X\colon \bF(X) \to 
\eexp(\widehat {\bL}(X))$ by $\kappa _X(x)=x$ for $x \in X$. 
Standard commutator calculus \cite{L} shows that
\begin{equation} 
\label{eq=malfree}
\gr^*(\kappa _X)\colon \gr^*(\bF(X)) \otimes \K  \isom 
\gr^*_F(\widehat {\bL}(X))
\end{equation}
is an isomorphism. It follows from \cite[Appendix A]{Q} that 
$\kappa _X$ is a Malcev completion. 

Now let $G$ be a finitely presented group, with presentation
$G= \langle x_1,\dots ,x_s  \mid w_1,\dots ,w_r \rangle$, or, 
for short, $G=\bF(X)/\langle { \bf w} \rangle$.  Denote by 
$\llangle{ \bf w}\rrangle$ 
the closed Lie ideal of $\widehat {\bL}(X)$ generated by 
$\kappa _X(w_1), \dots ,$ $ \kappa _X(w_r)$, and consider 
the group morphism induced by $\kappa _X$,
\begin{equation} 
\label{eq=malgen}
\kappa _G\colon G \to 
\eexp \big(\widehat {\bL}(X)/\llangle{ \bf w}\rrangle\big) .
\end{equation}
It follows from \cite{P1} that $\kappa_G$ is a Malcev completion 
for $G$. (For the purposes of that paper, it was assumed that 
$G_{\ab}$ had no torsion, see  \cite[Example 2.1]{P1}. 
Actually, the proof of the Malcev completion property 
applies verbatim in the general case, see \cite[Theorem 2.2]{P1}.)

\begin{prop} 
\label{prop=products}
If $G_1$ and $G_2$ are finitely presented $1$-formal groups, then their 
coproduct $G_1 *G_2$ and their product $G_1 \times G_2$ 
are again $1$-formal groups.
\end{prop}

\begin{proof}
First consider two arbitrary finitely presented groups, 
with presentations  $G_1= \bF(X)/\langle{ \bf u}\rangle$ 
and $G_2= \bF(Y)/\langle{ \bf v}\rangle $.  Then 
$G_1 *G_2=\bF(X \cup Y)/\langle {\bf u}, {\bf v} \rangle $. 
It follows from \eqref{eq=malgen} that  
$E_{G_1 *G_2}=E_{G_1 } \coprod   E_{G_2}$, 
the coproduct Malcev Lie algebra. 

On the other hand, $G_1 \times G_2=\bF(X \cup Y)/
\langle{ \bf u}, { \bf v}, (x,y); x \in X, y \in Y \rangle$, 
and so, by the same reasoning, 
$E_{G_1 \times G_2}= 
\widehat {\bL}(X \cup Y)/\llangle\kappa_X({ \bf u}), 
\kappa _Y({ \bf v}), (x,y); x \in X, y \in Y  \rrangle$.  
Using the Campbell-Hausdorff formula, we may 
replace each CH-group commutator $(x,y)$ with the 
corresponding Lie bracket, $[x,y]$; see 
\cite[Lemma 2.5]{P2} for details. We conclude that
$E_{G_1 \times G_2}= E_{G_1 } \prod   E_{G_2}$, 
the product Malcev Lie algebra.

Now assume  $G_1$ and $G_2$ are $1$-formal.  
Hence, we may write 
$E_{G_1 }= \widehat {\bL}(X')/\llangle{ \bf u'}\rrangle $
and $E_{G_2 }= \widehat {\bL}(Y')/\llangle{ \bf v'}\rrangle$, 
where the defining relations ${ \bf u'}$ and ${ \bf v'}$ are 
quadratic. Therefore
\begin{align*} 
\label{eq=coprod}
E_{G_1* G_2}& =\widehat {\bL}(X'\cup Y')/
\llangle{ \bf u'},{ \bf v'} \rrangle ,
\\
E_{G_1 \times G_2}&= \widehat {\bL}(X'\cup Y')/
\llangle{ \bf u'},{ \bf v'},[x',y']; x' \in X', y' \in Y' \rrangle  .
\end{align*}
Since the relations in these presentations are clearly quadratic, 
the $1$-formality of both $G_1* G_2$ and $G_1 \times G_2$ 
follows.  
\end{proof}

\subsection{Products, coproducts, and resonance}
\label{subsec=cpr}

Let $U^{i}$, $V^{i}$ ($i=1,2$) be complex vector spaces, where 
$U^1$ and $V^1$ are finite-dimensional.  Given two $\C$-linear 
maps, $\mu _U\colon U^1 \wedge  U^1 \to U^2$ 
and  $\mu _V\colon V^1 \wedge  V^1 \to V^2$, 
set $W^i=U^i \oplus V^i$, and define 
$\mu _U * \mu _V\colon W^1 \wedge  W^1 \to W^2$ 
as follows:  
\[
{\mu _U * \mu _V}\!\left|_{ U^1 \wedge  U^1} \right. = \mu_U, 
\quad
{\mu _U * \mu _V}\!\left|_{ V^1 \wedge  V^1} \right. = \mu_V, 
\quad
{\mu _U * \mu _V}\!\left|_{ U^1 \wedge  V^1} \right. = 0.
\]
When $\mu _U=\cup _{G_1}$ and $\mu _V=\cup _{G_2}$, 
then clearly $\mu _U * \mu _V=\cup _{G_1*G_2}$, since 
$K(G_1*G_2,1)=K(G_1,1) \vee K(G_2,1)$.

\begin{lemma} 
\label{lem=reswedge}
Suppose $U^1 \ne 0$,  $V^1 \ne 0$, and  $\mu_U * \mu_V$ 
satisfies the isotropicity resonance obstruction, i.e.,
each irreducible component of $\R_1(\mu_U * \mu_V)$ is a 
$p$-isotropic subspace of $W^1$, in the sense of
Definition \ref{def=position}.  Then $\mu_U= \mu_V=0$.
\end{lemma}

\begin{proof}
Set $\mu:=\mu _U * \mu _V$.  
We know $\R_1(\mu)=W^1$, by \cite[Lemma 5.2]{PS1}. 
If $\mu \ne 0$, then
$\mu$ is $1$-isotropic, with  $1$-dimensional 
image. It follows that either  $\mu _U= 0$ or $\mu _V=0$.
In either case, $\mu$ fails to be non-degenerate, a contradiction. 
Thus, $\mu=0$, and so $\mu _U= \mu _V=0$.
\end{proof}

Next, given $\mu_U$ and  $\mu _V$ as above, set 
$Z^1=U^1 \oplus V^1$ and 
$Z^2=U^2 \oplus V^2 \oplus (U^1 \otimes V^1)$, 
and define  
$\mu _U \times \mu _V\colon Z^1 \wedge  Z^1 \to Z^2$ 
as follows. As before, the restrictions of $\mu _U \times \mu _V$ to 
$U^1 \wedge  U^1$ and $V^1 \wedge  V^1$ are given by 
$\mu _U$ and  $\mu _V$,  respectively. On the other hand,  
$\mu _U \times \mu _V(u \wedge v)=u \otimes v$, for 
$u \in U^1$ and $v \in  V^1$.  
Finally, if $\mu _U=\cup _{G_1}$ and $\mu _V=\cup _{G_2}$, 
then $\mu _U \times \mu _V=\cup _{G_1 \times G_2}$, since 
$K(G_1\times G_2,1)=K(G_1,1) \times K(G_2,1)$.

\begin{lemma} 
\label{lem=resprod}
With notation as above, 
$\R_1(\mu _U \times \mu _V)=\R_1(\mu _U)\times \{0\} 
\cup\{0\}\times  \R_1(\mu _V)$.
\end{lemma}

\begin{proof}
Set $\mu=\mu_U \times \mu_V$.  The inclusion $\R_1(\mu)
\supset \R_1(\mu _U)\times \{0\} \cup\{0\}\times  \R_1(\mu _V)$ 
is obvious.  To prove the other inclusion, assume 
$\R_1(\mu) \ne 0$ (otherwise, there is nothing to prove), 
and pick $0\ne a+b \in \R_1(\mu)$, with $a \in U^1$ and 
$b \in V^1$. By definition of $\R_1(\mu)$, there is 
$x+y \in U^1 \oplus V^1$ such that 
$(a+b) \wedge (x+y) \ne 0$ and  
\begin{equation} 
\label{eq=muvanish}
\mu ((a+b) \wedge (x+y))= \mu_U(a\wedge x)+
\mu_V(b  \wedge y)+ a \otimes y -x  \otimes b=0 .
\end{equation}
In particular, $a \otimes y =x  \otimes b$. There are several 
cases to consider.

If $a \ne 0$ and $b\ne 0$, we must have $x=\lambda a$ 
and $y=\lambda b$, for some $\lambda \in \C$, and so 
$(a+b) \wedge (x+y) =(a+b)\wedge \lambda (a+b)=0$, 
a contradiction. 

If $b =0$, then $a \ne 0$ and \eqref{eq=muvanish}
forces $y=0$ and $\mu_U(a  \wedge x)=0$. 
Since $(a+b) \wedge (x+y)=a \wedge x \ne 0$, it follows
that $a \in \R_1(\mu _U)$, as needed. The other case, $a=0$, 
leads by the same reasoning to $b \in  \R_1(\mu _V)$.
\end{proof}

If $G_1$ and  $G_2$ are finitely generated groups, 
Lemma \ref{lem=resprod} implies that 
$\R_1(G_1 \times G_2)= \R_1(G_1) \times \{0\} \cup \{0\}\times \R_1(G_2)$.
An analogous formula holds for the characteristic varieties: 
$\V_1(G_1 \times G_2)=\V_1(G_1) \times \{1\}
\cup \{1\}\times \V_1(G_2)$, see \cite[Theorem 3.2]{CS2}. 

\subsection{Quasi-projectivity of coproducts} 
\label{subsec=vs}

Here is an application of Theorem \ref{thm=posobstr}. 
It is inspired by a result of M.~Gromov, who proved in 
\cite{G} that no non-trivial free product of groups can 
be realized as the fundamental group of a compact 
K\"{a}hler manifold.  We need two lemmas.

\begin{lemma} 
\label{lem=cup0}
Let $G$ be a finitely presented, commutator relators group
(that is, $G=\bF(X)/\langle {\bf w} \rangle $, with $X$ and 
${\bf w}$ finite, and ${\bf w}\subset \Gamma_2 \bF(X)$). 
Suppose $G$ is $1$-formal, and $\cup _{G}=0$.  
Then $G$ is a free group. 
\end{lemma} 

\begin{proof}
Pick a presentation $G=\bF(X)/\langle {\bf w} \rangle $, with 
all relators $w_i$ words in the commutators $(g,h)$, where 
$g,h \in \bF(X)$. We have 
$E_G=\widehat{\h} (G)$, by the $1$-formality of $G$, 
and $\h(G)=\bL (X)$, by the vanishing of $\cup _{G}$.  
Hence, $E_G= \widehat {\bL}(X)$. 
On the other hand, \eqref{eq=malgen} implies 
$E_G= \widehat {\bL}(X)/\llangle{\bf w} \rrangle$. 
We thus obtain a filtered Lie algebra isomorphism,  
$\widehat {\bL}(X) \xrightarrow{\,\simeq\,} 
\widehat {\bL}(X)/\llangle{\bf w} \rrangle$.

Taking quotients relative to the respective Malcev filtrations 
and comparing vector space dimensions, we see that 
$\kappa_X (w_i) \in \bigcap_{k \geq 1}F_k\widehat {\bL}(X)=0$, 
for all $i$. A well-known result of Magnus (see \cite{MKS}) says 
that $\gr^*(\bF(X))$ is a torsion-free graded abelian group. 
We infer from \eqref{eq=malfree} that 
$w_i \in  \bigcap_{k \geq 1} \Gamma_k\bF(X)$, for all $i$. 
Another well-known result of Magnus  (see \cite{MKS})  insures 
that $\bF(X)$ is residually nilpotent, i.e., 
$\bigcap_{k \geq 1} \Gamma_k\bF(X)=1$. 
Hence, $w_i=1$,  for all $i$, and so $G= \bF(X)$.
\end{proof}

\begin{lemma}
\label{lem=vfox}
Let $G_1$ and  $G_2$ be finitely presented groups with 
non-zero first Betti number. Then 
$\V_1(G_1 * G_2)=\T_{G_1 * G_2}$.
\end{lemma}

\begin{proof}
Let $G= \langle x_1, \dots, x_s \mid  w_1, \dots, w_r \rangle$ be 
an arbitrary finitely presented group, and let $\rho\in \T_G$ be an
arbitrary character. By Fox calculus, we know that $\rho\in \V_1(G)$ 
if and only if $b_1(G, \rho)>0$, where
$b_1(G, \rho) := \dim \ker d_1(\rho)- \rank d_2(\rho)$. 
Moreover, the linear map $d_1(\rho)\colon \C^s \to \C$ 
sends the basis element corresponding to the generator 
$x_i$ to $\rho (x_i)-1$, while the linear map 
$d_2(\rho)\colon \C^r \to \C^s$ is given by the 
evaluation at $\rho$ of the matrix of free derivatives 
of the relators, 
$\big( \frac{\partial w_j}{\partial x_i} (\rho) \big)$; 
see Fox \cite{Fox}.

For $G=G_1 * G_2$, write $\rho= (\rho_1, \rho_2)$, 
with $\rho_i \in \T_{G_i}$. We then have 
$d_j(\rho)= d_j(\rho_1) +d_j(\rho_2)$,  for $j=1, 2$. Hence, 
$b_1(G, \rho)= b_1(G_1, \rho_1) +  b_1(G_2, \rho_2) +1$, if both
$\rho_1$ and $\rho_2$ are different from $1$, and otherwise
$b_1(G, \rho)= b_1(G_1, \rho_1) +  b_1(G_2, \rho_2)$. Since
$b_1(G_i, 1)= b_1(G_i)>0$, the claim follows. 
\end{proof}

\begin{theorem} 
\label{thm=nonalgw}
Let $G_1$ and  $G_2$ be finitely presented groups 
with non-zero first Betti number.

\begin{enumerate}
\item \label{w1}
If the coproduct $G_1 * G_2$ is quasi-K\"{a}hler, then 
$\cup_{G_1}=\cup_{G_2}=0$.

\item \label{w2}
Assume moreover that $G_1$ and  $G_2$ are $1$-formal, 
presented by commutator relators only. Then 
$G_1 * G_2$ is a quasi-K\"{a}hler group if and only
if both $G_1$ and  $G_2$ are free.
\end{enumerate}
\end{theorem}

\begin{proof}
Part \eqref{w1}. Set $G= G_1 * G_2$. From Lemma \ref{lem=vfox}, 
we know that there is just one irreducible component of $\V_1(G)$ 
containing $1$, namely $\V =\T_G^{0}$, the component of the identity
in the character torus. Hence, $T_1(\V)= H^1(G, \C)$. Libgober's 
result from \cite{Li} implies then that $\R_1(G)=H^1(G, \C)$. 
If $G$ is quasi-K\"{a}hler, Theorem \ref{thm=posobstr}\eqref{a1} 
may be invoked to infer that $\cup_G$ satisfies the isotropicity 
resonance obstruction. The 
conclusion follows from Lemma \ref{lem=reswedge}.

Part \eqref{w2}. 
If $G_1$ and $G_2$ are free, then $G_1 * G_2$ is also free 
(of finite rank), thus quasi-projective. For the converse, 
use Part \eqref{w1} to deduce that $\cup _{G_1}=\cup _{G_2}=0$, 
and then apply Lemma \ref{lem=cup0}. 
\end{proof}

Let $\CC$ be the class of fundamental groups of complex 
projective curves of non-zero genus. Each $G \in \CC$ is a 
$1$-formal group, admitting a presentation with a single 
commutator relator, and is not free (for instance, since 
$\cup _G \ne 0$).  Proposition \ref{prop=products} and 
Theorem \ref{thm=nonalgw} yield the following corollary.  

\begin{corollary} 
\label{cor=cop1}
If $G_1,G_2 \in \CC$, then $G_1 *G_2$ is 
a $1$-formal group, yet $G_1 *G_2$ is not realizable as 
the fundamental group of a smooth, quasi-projective variety $M$. 
\end{corollary}

This shows that $1$-formality and quasi-projectivity may 
exhibit contrasting behavior with respect to the coproduct 
operation for groups.

\section{Configuration spaces} 
\label{sec=conf}

Denote by $S^{\times n}$ the $n$-fold cartesian product of a 
connected space $S$.  Consider the {\em configuration space}\/ 
of $n$ distinct labeled points in $S$,
\[
F(S, n)= S^{\times n}\setminus \bigcup_{i<j} \Delta_{ij} ,
\]
where $\Delta_{ij}$ is the diagonal $\{s\in S^{\times n} \mid  s_i=s_j\}$. 
The topology of configuration spaces has attracted considerable 
attention over the years.   For $S$ a smooth, complex projective 
variety, the cohomology algebra $H^*(F(S, n), \C)$ has been 
described by Totaro \cite{Tt}, solely in terms of $n$ and the 
cohomology algebra $H^*(S, \C)$. 

Let $C_g$ be a smooth compact complex curve of genus $g$ 
($g\ge 1$).  The fundamental group of the configuration space 
$M_{g,n}:= F(C_g, n)$ may be identified with $P_{g,n}$, the 
pure braid group on $n$ strings of the underlying Riemann surface.  
Starting from Totaro's description, it is straightforward to check 
that the low-degrees cup-product map of $P_{g,n}$ is equivalent, 
in the sense of Definition \ref{def=similar}, to the composite 
\begin{equation} 
\label{eq:cupg}
\mu_{g,n}\colon 
\xymatrix{ \bigwedge^2 H^1(C_g^{\times n},\C)
\ar[r]^(.55){\cup_{C_g^{\times n}}} & 
H^2(C_g^{\times n},\C)\ar@{>>}[r] &  
H^2(C_g^{\times n},\C)/ \spn \{[\Delta_{ij}]\}_{ i<j} } ,
\end{equation}
where $[\Delta_{ij}]\in H^2(C_g^{\times n}, \C)$ denotes the 
dual class of the diagonal $\Delta_{ij}$, and the second arrow 
is the canonical projection.  It follows that the connected  
smooth quasi-projective complex variety $M_{g,n}$ 
has the property that $W_1(H^1(M_{g,n},\C))=H^1(M_{g,n},\C)$, 
for all $g, n\ge 1$.

The Malcev Lie algebra of $P_{g,n}$ has been computed by 
Bezrukavnikov in \cite{B}, for all $g,n\ge 1$.  It turns out that 
the groups $P_{g,n}$ are $1$-formal, for $g>1$ and $n\ge 1$, 
or $g=1$ and $n\le 2$; see \cite[p.~130]{B}.   
On the other hand, Bezrukavnikov also states 
in \cite[Proposition 4.1(a)]{B} that $P_{1,n}$ is 
not $1$-formal for $n\ge 3$, without giving an argument.  
With our methods, this can be easily proved.

\begin{example} 
\label{ex:g1not1f}
Let  $\{ a, b\}$ be the standard basis of $H^1(C_1,\C)=\C^2$. 
Note that the cohomology algebra 
$H^*( C_1^{\times n}, \C)$ is isomorphic to 
$\bigwedge^* (a_1, b_1,\dots , a_n, b_n)$. 
Denote by $(x_1, y_1, \dots $, $x_n ,y_n)$ the 
coordinates of $z\in H^1(P_{1,n}, \C)$. Using \eqref{eq:cupg}, 
it is readily seen that
\renewcommand{\arraystretch}{1.1}
\begin{equation*} 
\label{eq:resg1}
\R_1(P_{1,n})=\left\{ (x,y) \in \C^n\times \C^n \left|
\begin{array}{l}
\sum_{i=1}^n x_i=\sum_{i=1}^n y_i=0 ,\\
x_i y_j-x_j y_i=0,  \text{ for $1\le i<j< n$}
\end{array}
\right\}. \right.
\end{equation*}
\renewcommand{\arraystretch}{1.0}

Suppose $n\ge 3$.  Then $\R_1(P_{1,n})$ 
is a rational normal scroll in $\C^{2(n-1)}$, 
see \cite{Har}, \cite{E}. In particular, $\R_1(P_{1,n})$ is an 
irreducible, non-linear variety.  From Theorem \ref{thm=bnew}\eqref{bn1}, 
we conclude that $P_{1,n}$ is indeed non-$1$-formal. 
This indicates that Theorem 1.3 from \cite{Hai} 
cannot hold in the stated generality.

This family of examples also shows that both the 
$\R_1$--version of Arapura's result on $\V_1$ from 
Theorem \ref{thm:vadm}\eqref{va1} and the isotropicity 
resonance obstruction may fail, for an arbitrary smooth 
quasi-projective variety $M$. 
\end{example}

For $n\le 2$, things are even simpler. 

\begin{example} 
\label{ex:01iso}
It follows from \eqref{eq:cupg} that $\mu_{1,2}$ equals the 
canonical projection
\begin{equation*} 
\label{eq:cup12}
\mu_{1,2} \colon \bigwedge\nolimits^2 (a_1,b_1, a_2,b_2) \surj 
\bigwedge\nolimits^2 (a_1,b_1, a_2,b_2)/ \C \cdot 
(a_1-a_2)(b_1-b_2) .
\end{equation*}
It follows that $\R_1(P_{1,2})$ is a $2$-dimensional, 
$0$-isotropic linear subspace of $H^1 (P_{1,2},\C)$.

Consider now the smooth variety $M'_g:= M_{1,2}\times C_g$, 
with $g\ge 2$. By Proposition \ref{prop=products}, this 
variety has $1$-formal fundamental group. It also has 
the property that $W_1(H^1 (M'_g,\C))=H^1 (M'_g,\C)$. 
We infer from Lemma \ref{lem=resprod} that
\[
\R_1(\pi_1(M'_g))=\R_1(P_{1,2})\times\{0\} \cup
\{0\}\times  H^1 (C_g,\C) ,
\]
where the component $\R_1(P_{1,2})$ is $0$-isotropic 
and the component $H^1 (C_g,\C)$ is $1$-isotropic. 
We thus see that both cases listed in Proposition 
\ref{prop=posobstr}\eqref{ra1} may actually occur.
\end{example}

\begin{remark}
\label{rem:lib}
Recall from Example \ref{ex:notinj} that the tangent cone formula 
may fail for quasi-projective groups, at least in the case 
when $1$ is an isolated point of the characteristic variety.
The following statement can be extracted from \cite[p.~161]{Li}:
``If $M$ is a quasi-projective variety and $1$ is not an isolated 
point of $\V_1(\pi_1(M))$, then $TC_1(\V_1(\pi_1(M)))= 
\R_1(\pi_1(M))$."   Taking $M$ to be one of the configuration 
spaces $M_{1,n}$, with  $n\ge 3$, shows that this statement 
does not hold, even when $M$ is smooth.

Indeed, since $P_{1,2}$ is $1$-formal, we obtain from 
Theorem \ref{thm=tcfintro} that $\V_1(P_{1,2})$ is 
$2$-dimensional at $1$. As is well-known,
the natural surjection, $P_{1,n}\surj P_{1,2}$, embeds
$\V_1(P_{1,2})$ into $\V_1(P_{1,n})$, for $n\ge 2$. Thus, 
$\V_1(P_{1,n})$ is positive-dimensional at $1$, for $n\ge 2$.
On the other hand, it follows from Example \ref{ex:g1not1f} that
$TC_1(\V_1(P_{1,n}))$ is strictly contained in $\R_1(P_{1,n})$, 
for $n\ge 3$.
\end{remark}

\section{Artin groups} 
\label{sec=artinalg}

In this section, we analyze the class of finitely generated Artin 
groups. Using the resonance obstructions from 
Theorem \ref{thm=posobstr}, we give a complete 
answer to Serre's question for right-angled Artin groups, 
and we give a Malcev Lie algebra version of the answer for 
arbitrary Artin groups. 

\subsection{Labeled graphs and Artin groups}
\label{subsec=artin}
Let $\Gamma=(\sV,\sE,\ell)$ be a labeled finite simplicial graph, 
with vertex set $\sV$, edge set $\sE\subset \binom{\sV}{2}$, 
and labeling function $\ell \colon \sE\to \N _{\geq 2}$.  
Finite simplicial graphs are identified in the sequel with labeled 
finite simplicial graphs with $\ell (e)=2$, for each $e \in\sE$.

\begin{definition} 
\label{def=artgps} 
The {\em Artin group}\/ associated to the labeled graph $\Gamma$ is 
the group $G_{\Gamma}$ generated by the vertices $v \in \sV$ and 
with a defining relation 
\[
\underbrace{vwv\cdots}_{\ell(e)} =\underbrace{wvw\cdots}_{\ell(e)}
\]
for each edge $e=\{v,w\}$ in $\sE$.  If $\Gamma$ is unlabeled, 
then $G_{\Gamma}$ is called a {\it right-angled Artin group}, and 
is defined by commutation relations $vw=wv$,
one for each edge $\{v,w\} \in\sE$.
\end{definition}

\begin{example}
\label{ex=artjoin}
Let $\Gamma=(\sV,\sE,\ell)$ and $\Gamma'=(\sV',\sE',\ell')$ be two 
labeled graphs.  Denote by $\Gamma \bigsqcup \Gamma '$ their 
disjoint union, and by $\Gamma *\Gamma '$ their {\em join}, 
with vertex set $\sV \bigsqcup \sV '$, edge set 
$ \sE \bigsqcup \sE ' \bigsqcup \{  \{ v,v'\} \mid  v \in \sV, v' \in \sV'\}$, 
and label $2$ on each edge $ \{ v,v'\}$.  Then
\[
G_{\Gamma \bigsqcup \Gamma '}  = G_{\Gamma} *G_{\Gamma '}
\quad \text{and} \quad 
G_{\Gamma *\Gamma '}  = G_{\Gamma} \times G_{\Gamma '}\, .
\]
In particular, if $\Gamma$ is a discrete graph, i.e., $\sE=\emptyset$, 
then $G_{\Gamma}=\FF_n$, whereas if $\Gamma$ is an (unlabeled) 
complete graph, i.e., $\sE=\binom{\sV}{2}$, then 
$G_{\Gamma}=\Z^n$, where $n=\abs{\sV}$.  More generally, 
if $\Gamma$ is a complete multipartite graph  (i.e., a finite 
join of discrete graphs), then $G_{\Gamma}$ is a finite 
direct product of finitely generated free groups. 
\end{example}

Given a graph $\Gamma=(\sV,\sE)$ and a subset of vertices $\sW \subset \sV$, 
we denote by $\Gamma(\sW)$ the full subgraph of $\Gamma$, 
with vertex set  $\sW $ and edge set $\sE \cap \binom{\sW}{2} $.

Let $(S^1)^{\sV}$ be the compact $n$-torus, where $n=|\sV|$, endowed with 
the standard cell structure.  Denote by $K_{\Gamma}$ the subcomplex of 
$(S^1)^{\sV}$ having a $k$-cell for each subset $\sW \subset \sV$ of size $k$ 
for which $\Gamma(\sW)$ is a complete graph. As shown by Charney--Davis
\cite{CD} and Meier--VanWyk \cite{MV},  $K_{\Gamma}=K(G_{\Gamma},1)$. 
In particular, the cup-product map 
$\cup_{G_{\Gamma}}\colon H^1(G_{\Gamma},\C)\wedge H^1(G_{\Gamma},\C) 
\to H^2(G_{\Gamma},\C)$ may be identified with the linear map 
$\cup_{\Gamma}\colon \C^\sV \wedge  \C^\sV \to \C^\sE$ 
defined by
\begin{equation} 
\label{eq=cupart}
v\cup_{\Gamma}w = 
\begin{cases} 
\pm  \{v,w\} , &\text{if  $\{v,w\} \in \sE$},\\
0,  &\text{otherwise} ,
\end{cases}  
\end{equation}
with signs determined by fixing an orientation of the edges of $\Gamma$.

\subsection{Jumping loci for right-angled Artin groups}
\label{subsec=jumpartin}

The resonance varieties of right-angled Artin groups  
were  described explicitly in Theorem 5.5 from \cite{PS1}. 
If $\Gamma=(\sV,\sE)$ is a graph, then 
\begin{equation} 
\label{eq=resart}
\R_1(G_{\Gamma})= \bigcup_{\sW} \C^ \sW ,
\end{equation}
where the union is taken over all subsets $\sW \subset \sV$ 
such that $\Gamma(\sW)$ is disconnected, and maximal with 
respect to this property.   Moreover, the decomposition \eqref{eq=resart} 
coincides with the decomposition into irreducible components of 
$\R_1(G_{\Gamma})$.

Before proceeding to the Serre problem, we describe the 
characteristic variety of $G_{\Gamma}$. 
For $\sW \subset \sV$, define the subtorus 
$\T _{\sW} \subset \T _{G_{\Gamma }   }=(\C^*)^{\sV}$ by
\[
\T _{\sW}=\{(t_v)_{v \in \sV} 
\in (\C^*)^{\sV} \mid  \text{$t_v=1$ for $v \notin \sW$} \} .
\]
The  map 
$\exp\colon T_1\T_{G_\Gamma} \to \T_{G_{\Gamma}}$ 
is the componentwise  exponential map 
$\exp^{\sV}\colon \C^\sV\to(\C^*)^{\sV}$; 
its restriction to the subspace spanned by $\sW$ 
is $\exp^{\sW}\colon \C^\sW\to(\C^*)^{\sW}=\T_{\sW}$.

\begin{prop} 
\label{prop=vart}
Let $G_{\Gamma }$ be the right-angled Artin group associated 
to the graph $\Gamma=(\sV,\sE)$. Then
\[
\V_1(G_{\Gamma })= \bigcup _{\sW}\T _{\sW} ,
\]
where the union is over all subsets $\sW\subset \sV$ 
such that $\Gamma(\sW)$ is maximally disconnected.   
Moreover, this decomposition coincides with the 
decomposition into irreducible components of 
$\V_1(G_{\Gamma})$.

\end{prop}

\begin{proof}
The realization of $K(G_{\Gamma },1)$ as a subcomplex $K_{\Gamma }$ 
of the torus $(S^1)^{\sV}$ yields a well-known resolution of the trivial 
$\Z G_{\Gamma}$-module $\Z$ by finitely generated, free 
$\Z G_{\Gamma}$-modules, as the augmented, 
$G_{\Gamma}$-equivariant chain complex of the 
universal cover of $K_{\Gamma }$,
\[
{\widetilde C}_{\bullet} (\widetilde{K_{\Gamma}}) \colon 
\xymatrixcolsep{18pt}
\xymatrix{\cdots \ar[r] &  \Z G_{\Gamma } \otimes C_k  
\ar^(.45){d_k }[r]& \Z G_{\Gamma } \otimes C_{k-1}  
\ar[r] & \cdots \ar^(.45){d_1}[r] & \Z G_{\Gamma } 
\ar^{\epsilon}[r] & \Z \ar[r] & 0  }. 
\]
Here $C_k$ denotes the free abelian group generated by the 
$k$-cells of $K_{\Gamma }$, and the boundary maps are given by 
\begin{equation} 
\label{eq=kchains}
d_k(e_{v_1} \times \cdots \times e_{v_k})=
\sum _{i=1}^{k}(-1)^{i-1}(v_i-1) \otimes 
e_{v_1} \times \cdots \times {\widehat  e_{v_i} } 
\times \cdots \times e_{v_k} ,
\end{equation}
where, for each $v \in \sV$, the symbol $e_v$ denotes   
the $1$-cell  corresponding to the $v$-th factor of 
$(S^1)^{\sV}$.

Now let $\rho= (t_v)_{v \in \sV} \in (\C^*)^{\sV}$ be an arbitrary character. 
Denoting by $\{v^*\} _{v \in \sV}$ the basis of $H^1(G_{\Gamma },\C)$ 
dual to the canonical basis of $H_1(G_{\Gamma },\C)$, define 
an element $z \in \C^\sV= H^1(G_{\Gamma },\C)$ by
$z=\sum_{v \in \sV}(t_v-1)v^*$.  Using \eqref{eq=kchains}, 
it is not difficult to check the following equality of cochain 
complexes
\begin{equation} 
\label{eq=reschar}
\Hom_{\Z  G_{\Gamma }}( {\widetilde C}_{\bullet} 
(\widetilde{K_{\Gamma}}),{}_{\rho}\C) = 
(H^{\bullet} (G_{\Gamma },\C), \mu_ z ) .
\end{equation}
It follows then, directly from the definitions  
and using \eqref{eq=reschar}, that 
$\rho \in \V_1(G_{\Gamma })$ if and only if 
$z \in \R_1(G_{\Gamma })$.  Hence, the claimed 
decomposition of $\V_1(G_{\Gamma })$ is a direct 
consequence of the decomposition \eqref{eq=resart}.
\end{proof}

\subsection{Serre's problem for right-angled Artin groups}
\label{subsec=rightartin}

As shown by Kapovich and Millson in \cite[Theorem 16.10]{KM}, 
all Artin groups are $1$-formal.  This opens the way for approaching 
Serre's problem for Artin groups via resonance varieties, which, 
as noted above, were described explicitly in \cite{PS1}.  Using 
these tools, we find out precisely which right-angled Artin groups 
can be realized as fundamental groups of quasi-compact 
K\"{a}hler manifolds.

\begin{theorem} 
\label{thm=raserre} 
Let  $\Gamma=(\sV,\sE)$ be a finite simplicial graph, with associated 
right-angled Artin group $G_{\Gamma}$. The following are equivalent.

\begin{romenum}
\item  \label{s1} 
The group $G_{\Gamma}$ is quasi-K\"{a}hler.
\item  \label{s2} 
Every positive-dimensional irreducible component 
$\R^{\alpha}$ of $\R_1(\cup_{G_{\Gamma }})$
is a $p$-isotropic linear subspace of $H^1(G_{\Gamma}, \C)$, 
of dimension at least $2p+2$, for some 
$p=p(\alpha) \in \{0,1\}$. 
\item  \label{s3} 
The graph  $\Gamma$ is complete multipartite graph.
\item  \label{s4}  
The group $G_{\Gamma}$ is a finite product of finitely generated 
free groups.
\end{romenum}
\end{theorem}

\begin{proof}
For the implication \eqref{s1} $\Rightarrow$ \eqref{s2}, use 
the $1$-formality of $G_{\Gamma}$ and 
Theorem \ref{thm=posobstr}.

\smallskip
The implication  \eqref{s2} $\Rightarrow$ \eqref{s3} 
is proved by induction on $n=\abs{\sV}$. If  $\Gamma$ is complete, 
then  $\Gamma$ is the join of $n$ graphs with one vertex. Otherwise, 
there is a subset $\sW \subset \sV$ such that $\Gamma (\sW)$ is 
disconnected, and maximal with respect to this property.  Write 
$\sW = \sW' \bigsqcup   \sW''$, with both $ \sW'$ and $ \sW''$ 
non-empty and with no edge connecting $ \sW'$ to $ \sW''$. 
Then $\Gamma (\sW)=\Gamma (\sW ')  \bigsqcup \Gamma (\sW '')$, 
and so $G_{\Gamma (\sW)}=G_{\Gamma (\sW ')} * 
G_{\Gamma (\sW '')} $.  Hence, $w' \cup_{\Gamma (\sW)} w''=0$, 
for any $w' \in \sW '$ and $w'' \in \sW''$.  We infer from  
\cite[Lemma 5.2]{PS1} that $\R_1(G_{\Gamma (\sW)})=\C^\sW$. 

On the other hand, we know from \eqref{eq=resart} that $\C^\sW$
is a positive-dimensional irreducible component of 
$\R_1(G_{\Gamma })$. Our hypothesis implies that 
$\C^\sW$ is either $0$-isotropic  or $1$-isotropic with 
respect to $\cup_{\Gamma (\sW)}$.  By Lemma \ref{lem=reswedge}, 
$ \cup _{\Gamma (\sW ')}=\cup _{\Gamma (\sW '')}  =0$.  
The cup-product formula \eqref{eq=cupart} implies that 
$\Gamma (\sW)$ is a discrete graph.

If $\sW=\sV$, we are done. Otherwise, $\sV=\sW \bigsqcup 
\sW_1$, with $\abs{\sW_1} <n$. Since  $\Gamma (\sW)$ is maximally 
disconnected, this forces $\{ w, w_1\} \in \sE$, for all $w \in \sW$ 
and $w_1 \in \sW_1$. In other words, $\Gamma $ is the join 
$\Gamma (\sW) * \Gamma (\sW_1)$; thus, 
$G_{\Gamma }=G_{\Gamma (\sW)} \times G_{\Gamma (\sW_1)}$. 
By Lemma \ref{lem=resprod}, $\cup_{\Gamma (\sW_1)}$ inherits 
from $\cup_{\Gamma}$ the isotropicity property.  
This completes the induction.

\smallskip
The implication  \eqref{s3} $\Rightarrow$ \eqref{s4} follows 
from the discussion in Example \ref{ex=artjoin}.

\smallskip
Finally, the implication  \eqref{s4} $\Rightarrow$ \eqref{s1} 
follows by taking products and realizing free groups by the 
complex line with a number of points deleted.
\end{proof}

As is well-known, two right-angled Artin groups are isomorphic 
if and only if the corresponding graphs are isomorphic. 
Evidently, there are infinitely many graphs which are not 
joins of discrete graphs.  Thus, implication \eqref{s1}  
$\Rightarrow$ \eqref{s3} from Theorem \ref{thm=raserre} 
allows us to recover, in sharper form, a result of Kapovich 
and Millson (Theorem 14.7 from \cite{KM}).

\begin{corollary}
\label{cor=rkm}
Among right-angled Artin groups $G_{\Gamma}$, there are infinitely 
many mutually non-isomorphic groups which are not isomorphic 
to fundamental groups of smooth, quasi-projective complex varieties. 
\end{corollary}

\subsection{A Malcev Lie algebra version of Serre's question}
\label{subsec=mlserre}

Next, we describe a construction that associates to a labeled 
graph $\Gamma=(\sV,\sE,\ell)$ an ordinary graph, 
${\tilde \Gamma}=({\tilde  \sV}, {\tilde \sE})$, which 
we call the {\em odd contraction}\/ of $\Gamma$.
First define $\Gamma_{\odd}$ to be the unlabeled 
graph with vertex set $\sV$ and edge set 
$\{e \in \sE \mid  \ell (e)~  \text{is odd}\}$.
Then define ${\tilde  \sV}$ to be the set of connected 
components of $\Gamma_{\odd}$, with two distinct 
components determining an edge $\{c,c'\} \in  {\tilde \sE}$ 
if and only if there exist vertices $v\in c$ 
and $v'\in c'$ which are connected by an edge in $\sE$.

\begin{example} 
\label{ex=braids} 
Let $\Gamma$ be the complete graph on vertices 
$\{1,2,\dots ,n-1\}$, with labels $\ell(\{i,j\})=2$ if $|i-j|>1$ and  
$\ell(\{i,j\})=3$ if $|i-j|=1$. The corresponding Artin group 
is the classical  braid group on $n$ strings, $B_n$. 
Since in this case  $\Gamma_{\odd}$ is connected, the odd 
contraction ${\tilde \Gamma}$ is the discrete graph with 
a single vertex.
\end{example}

\begin{lemma}
\label{lem=malcont}
Let $\Gamma=(\sV,\sE,\ell)$ be  a labeled graph, with odd 
contraction ${\tilde \Gamma}=({\tilde  \sV} , {\tilde \sE})$. 
Then the Malcev Lie algebra of $G_{\Gamma}$ is filtered Lie 
isomorphic to the Malcev Lie algebra of  $G_{\tilde \Gamma}$.
\end{lemma}

\begin{proof}
The Malcev Lie algebra of $G_{\Gamma }$ was computed in  
\cite[Theorem 16.10]{KM}. It is the quotient of the free 
Malcev Lie algebra on $\sV$, ${ \widehat \bL}(\sV)$, by 
the closed Lie ideal generated by the differences $u-v$, 
for odd-labeled edges $\{u,v\}\in \sE$, and by the 
brackets $[u,v]$, for even-labeled edges $\{u,v\}\in \sE$. 
Plainly, this quotient is filtered Lie isomorphic to the 
quotient of  ${ \widehat \bL}({\tilde \sV})$ by the closed 
Lie ideal generated by the brackets $[c,c']$, for 
$\{c,c'\} \in {\tilde \sE}$, which is just the Malcev Lie 
algebra of $G_{ {\tilde \Gamma}    }$. 
\end{proof}

The Coxeter group associated to a labeled graph 
$\Gamma=(\sV,\sE,\ell)$ is the quotient of the Artin group 
$G_{\Gamma}$ by the normal subgroup generated by 
$\{v^2 \mid v\in \sV\}$.  If the Coxeter group $W_{\Gamma}$ 
is finite, then $G_{\Gamma}$ is quasi-projective. 
The proof of this assertion, due to Brieskorn \cite{Br}, 
involves the  following steps:  $W_\Gamma$ acts as a 
group of reflections in some $\C^{n}$; the action is free 
on the complement $M_{\Gamma}$ of the arrangement 
of reflecting hyperplanes of $W_{\Gamma}$, and 
$G_\Gamma=\pi_1(M_{\Gamma}/W_\Gamma)$;   
finally, the quotient manifold $M_{\Gamma}/W_\Gamma$ 
is a complex smooth quasi-projective variety.

It would be interesting to know exactly which (non-right-angled) 
Artin groups can be realized by smooth, quasi-projective complex 
varieties.  We give an answer to this question, albeit only at the 
level of Malcev Lie algebras of the respective groups.

\begin{corollary} 
\label{cor=aserre}
Let $\Gamma$ be  a labeled graph, with associated 
Artin group $G_{\Gamma }$ and odd contraction the 
unlabeled graph ${\tilde \Gamma}$. The following are equivalent.

\begin{romenum}
\item \label{arba1}
The Malcev Lie algebra of $G_{\Gamma}$ is filtered Lie 
isomorphic to the Malcev Lie algebra of a quasi-K\"{a}hler group.

\item \label{arba2}
The isotropicity property from Theorem \ref{thm=raserre}\eqref{s2}
is satisfied by $\cup_{G_{\Gamma}}$.

\item  \label{arba3} 
The graph ${\tilde \Gamma}$ is a complete multipartite graph.

\item  \label{arba4}  
The Malcev Lie algebra of $G_{\Gamma}$ is filtered Lie 
isomorphic to the Malcev Lie algebra of a finite product of finitely 
generated free groups.
\end{romenum}
\end{corollary}

\begin{proof}
By Lemma \ref{lem=malcont}, the Malcev Lie algebras of 
$G_{\Gamma}$ and $G_{ {\tilde \Gamma} }$ are filtered isomorphic. 
Hence, the graded Lie algebras $\gr^*(G_{\Gamma }) \otimes \C$ 
and  $\gr^*(G_{ {\tilde \Gamma} }) \otimes \C$ are isomorphic.

From \cite{S}, we know that the kernel of 
the Lie bracket, $\bigwedge ^2 \gr^1(G) \otimes \C \to 
\gr^2(G) \otimes \C$, is equal to  $\im (\partial _G)$, 
for any finitely presented group $G$. It follows that the 
cup-product maps $\cup _{G_{\Gamma }}$ and 
$\cup _{G_{ {\tilde \Gamma}}}$ are equivalent, in the 
sense of Definition \ref{def=similar}.  Consequently, 
$\cup_{G_{\Gamma}}$ satisfies the isotropicity resonance 
obstruction if and only if $\cup_{G_{{\tilde\Gamma}}}$ 
does so.

With these remarks, the Corollary follows at once 
from Theorems \ref{thm=raserre} and \ref{thm=posobstr}.
\end{proof}

\subsection{K\"{a}hler right-angled Artin groups}
\label{subsec:kraag}

With our methods, we may easily decide which 
right-angled Artin groups are K\"{a}hler groups.  

\begin{corollary} 
\label{cor:raag kahler} 
For a right-angled Artin group $G_{\Gamma}$, 
the following are equivalent.
\begin{romenum}
\item  \label{agk1} 
The group $G_{\Gamma}$ is K\"{a}hler.
\item  \label{agk2} 
The graph  $\Gamma$ is a complete graph on an even 
number of vertices.
\item  \label{agk3}  
The group $G_{\Gamma}$ is a free abelian group of 
even rank. 
\end{romenum}
\end{corollary}

\begin{proof}
Implications \eqref{agk2} $\Rightarrow$ \eqref{agk3} 
$\Rightarrow$ \eqref{agk1}  are clear. So suppose 
$G_{\Gamma}$ is a K\"{a}hler group. By Theorem 
\ref{thm=raserre}, $\Gamma$ is a complete multi-partite graph 
$\overline{K}_{n_1}*\cdots * \overline{K}_{n_r}$, and 
$G_\Gamma=\bF_{n_1}\times \cdots \times \bF_{n_r}$.  
By Lemma \ref{lem=resprod}, and abusing notation 
slightly, $\R_1(G_{\Gamma })=\bigcup_i \R_1(\bF_{n_i})$.
Now, if there were an index $i$ for which $n_i>1$, then 
$\R_1(\bF_{n_i})=\C^{n_i}$ would be a positive-dimensional, 
$0$-isotropic, irreducible component of $\R_1(G_{\Gamma })$, 
contradicting Corollary \ref{cor:firstex}\eqref{appl1}. 
Thus, we must have $n_1=\cdots =n_r=1$, and $\Gamma=K_r$.  
Moreover, since $G_\Gamma=\Z^r$ is a K\"{a}hler group, 
$r$ must be even. 
\end{proof}

\begin{ack}
This work is a major revision of the preprint \cite{DPS05}.  
We are grateful to an anonymous referee, who suggested 
significant improvements to \cite{DPS05}, concerning both 
the substance of the results, and the exposition. These 
suggestions are reflected especially in Sections
\ref{sec=holo} and \ref{sec=reljump} of this paper.  
We also thank a second referee for further suggestions 
on the exposition.  
\end{ack}

\quad\newline\vspace{-1.1in}
\newcommand{\arxiv}[1]
{\texttt{\href{http://arxiv.org/abs/#1}{arXiv:#1}}}

\renewcommand{\MR}[1]
{\href{http://www.ams.org/mathscinet-getitem?mr=#1}{MR#1}}

\bibliographystyle{amsplain}

\end{document}